\documentclass[12pt,twoside]{amsart}
\usepackage{amssymb,amsmath,amsthm, amscd, enumerate, mathrsfs}
\usepackage{graphicx, hhline}
\usepackage[all]{xy}
\usepackage[usenames]{color}
\usepackage{hyperref}
\hypersetup{colorlinks=true}

\title{On semipositivity, injectivity and vanishing theorems}
\author{Osamu Fujino} 
\date{2016/9/27, version 0.12}
\keywords{semipositivity, nefness, 
vanishing theorem, injectivity theorem, 
canonical ring, pluricanonical divisor}

\subjclass[2010]{Primary 14F17; Secondary 14E30, 14D07}
\address{Department of Mathematics, Graduate School of Science, 
Osaka University, Toyonaka, Osaka 560-0043, Japan}
\email{fujino@math.sci.osaka-u.ac.jp}
\dedicatory{Dedicated to Professor Steven Zucker on the 
occasion of his 65th birthday}

\newcommand{\Supp}[0]{{\operatorname{Supp}}}
\newtheorem{thm}{Theorem}[section]
\newtheorem{lem}[thm]{Lemma}
\newtheorem{cor}[thm]{Corollary}
\newtheorem{prop}[thm]{Proposition}
\newtheorem{conj}[thm]{Conjecture}
\newtheorem*{claim}{Claim}

\theoremstyle{definition}
\newtheorem{defn}[thm]{Definition}
\newtheorem{problem}[thm]{Problem}
\newtheorem{rem}[thm]{Remark}
\newtheorem*{ack}{Acknowledgments} 
\newtheorem{say}[thm]{}

\newtheorem{ex}[thm]{Example}
\begin{document}

\begin{abstract}
This is a survey article on the recent developments of 
semipositivity, injectivity, and vanishing theorems for 
higher-dimensional complex projective 
varieties. 
\end{abstract}

\maketitle

\tableofcontents 

\section{Introduction} 

This paper is a survey article on the recent developments 
of semipositivity, injectivity, and 
vanishing theorems for higher-dimensional complex 
projective varieties (see, for example, \cite{fujino-vanishing}, 
\cite{fujino-injectivity}, \cite{fujino-foundation}, 
\cite{fujino-fujisawa}, and \cite{ffs}). 

We know that many important generalizations of 
the Kodaira vanishing theorem,
for example, the Kawamata--Viehweg vanishing theorem, Koll\'ar's 
injectivity, torsion-free, 
and vanishing theorems, the Nadel vanishing theorem, and so on, 
were obtained in 1980s. 
They have already played crucial roles in the study 
of higher-dimensional complex projective varieties. 
We note that the Fujita--Zucker--Kawamata semipositivity 
theorem for direct images of relative canonical bundles 
has also played important roles. One of my main 
motivations was to establish a more general 
cohomological package based on the theory of mixed 
Hodge structures on cohomology with compact support. 
Now I think that our new results are almost satisfactory (see 
Theorems \ref{f-thm2.12}, \ref{f-thm2.13}, and \ref{f-thm3.6}). 
They are waiting for applications. 
I hope that the reader would find various applications of 
our semipositivity, injectivity, and vanishing theorems.

Let us see the contents of this paper. 
In Section \ref{f-sec2}, we first discuss the Hodge 
theoretic aspect of 
Kodaira-type vanishing theorems (see, for example, 
\cite{esnault-viehweg}, \cite[Part III]{kollar-book}, 
\cite{fujino-vanishing}, \cite{fujino-injectivity}, 
and \cite{fujino-foundation}). 
I emphasize the importance of 
Koll\'ar's injectivity theorem and its generalizations. 
I think that 
one of the most important recent developments is the 
introduction of mixed Hodge structures on 
cohomology with compact support 
in order to generalize 
Koll\'ar's injectivity theorem (see, for example, 
\cite{fujino-fundamental}, \cite{fujino-vanishing}, 
\cite{fujino-injectivity}, and \cite{fujino-foundation}). 
Next we discuss Enoki's injectivity theorem, 
which is an analytic counterpart of 
Koll\'ar's injectivity theorem. 
I like Enoki's idea since it is very simple and powerful. 
Enoki's proof only uses the standard results 
of the theory of harmonic forms on compact K\"ahler manifolds. 
Although I obtained some generalizations of Enoki's 
injectivity theorem and their applications (see 
\cite{fujino-trans} and \cite{fujino-trans2}), 
I think that they are not satisfactory for most geometric applications. 

In Section \ref{f-sec3}, we treat several semipositivity theorems for 
direct images of relative (log) canonical bundles and 
relative 
pluricanonical bundles. 
The (numerical) semipositivity of direct images of 
relative (log) canonical bundles discussed in this paper is 
more or less Hodge theoretic. 
Note that mixed Hodge structures on cohomology with compact support 
are also very useful for semipositivity theorems. 
By considering their variations, 
we can prove a powerful semipositivity theorem 
by the theory of gradedly polarizable 
admissible variation of mixed Hodge structure (see \cite{fujino-fujisawa} and 
\cite{ffs}). 
Unfortunately, since I am not familiar with 
the recent developments of semipositivity theorems by $L^2$ methods, 
I do not discuss the analytic aspect of 
semipositivity theorems in this paper. 
In Subsection \ref{f-subsec3.1}, 
we explain new semipositivity theorems 
for direct images of relative 
pluricanonical bundles with the aid of 
the minimal model program (see \cite{fujino-direct}). 
I think that it is highly desirable 
to recover them without using the minimal model program. 

In Section \ref{f-sec4}, we will see that 
pluricanonical divisors sometimes 
behave much better than canonical divisors. 
We discuss two different topics. 
In Subsection \ref{f-subsec4.1}, 
we explain Koll\'ar's famous result on plurigenera in \'etale 
covers of smooth projective varieties of general type. 
We give Lazarsfeld's proof using 
the theory of asymptotic multiplier ideal sheaves for the 
reader's convenience and 
also a proof based on the minimal model program. 
The proof based on the minimal model program is harder than 
Lazarsfeld's proof but is interesting and natural 
from the minimal model theoretic viewpoint. 
In Subsection \ref{f-subsec4.2}, 
we explain Viehweg's ampleness theorem for direct images of 
relative pluricanonical bundles, which is buried in Viehweg's papers. 
I think that these results may help the reader to understand the reason 
why 
we should consider pluricanonical divisors for 
the study of higher-dimensional algebraic varieties. 

In Section \ref{f-sec5}, we quickly review the finite generation of (log) 
canonical rings due to Birkar--Cascini--Hacon--M\textsuperscript{c}Kernan. 
I want to emphasize that 
we need the semipositivity theorem discussed in Section \ref{f-sec3} 
when we treat (log) canonical rings for varieties which are not of 
(log) general type (see \cite{fujino-mori} and \cite{fujino-some}). We also explain 
the nonvanishing conjecture, which is one of the most important 
conjectures for higher-dimensional 
complex projective varieties. 

Section \ref{f-sec6} is an appendix, where 
we collect some definitions with the intention of helping the reader to understand 
this paper. The reader can read each section separately. 

\begin{ack}
The author was partially supported by Grant-in-Aid for 
Young Scientists (A) 24684002 and Grant-in-Aid for 
Scientific Research (S) 24224001 from JSPS. 
He thanks Professor Steven Zucker for useful comments 
and advice. 
He also thanks Professor Fabrizio Catanese for answering his questions and 
Yoshinori Gongyo for various discussions. 
Finally, he thanks Shin-ichi Matsumura for sending his preprints. 
\end{ack}

We will work over $\mathbb C$, the complex number 
field, throughout this paper. In this 
paper, a scheme means a separated scheme of finite type 
over $\mathbb C$. 

\section{On injectivity theorems and 
vanishing theorems}\label{f-sec2} 

I think that one of the most fundamental results for complex 
projective varieties is Koll\'ar's 
injectivity theorem (see Theorem \ref{f-thm2.1}). 
The importance of the Kawamata--Viehweg (or Nadel) vanishing theorem 
for the study of higher-dimensional complex algebraic varieties 
is repeatedly emphasized in many papers and textbooks (see, 
for example, \cite{kollar-mori} and \cite{lazarsfeld}). 
On the other hand, I think that the importance of 
Koll\'ar's injectivity theorem has not been 
emphasized so far in the standard literature. 

Let us recall Koll\'ar's injectivity theorem. 

\begin{thm}[{\cite[Theorem 2.2]{kollar-higher1}}]\label{f-thm2.1}
Let $X$ be a smooth projective variety and let $L$ be a semiample 
Cartier divisor on $X$, that is, 
the complete linear system 
$|mL|$ has no base points for some positive 
integer $m$. 
Let $D$ be a member of $|kL|$ for some positive integer $k$. 
Then 
$$
H^i(X, \mathcal O_X(K_X+lL))\to H^i(X, \mathcal O_X(K_X+(l+k)L)), 
$$ 
induced by the natural 
inclusion $\mathcal O_X\hookrightarrow\mathcal O_X(D)\simeq 
\mathcal O_X(kL)$, 
is injective for every $i$ and every positive integer $l$. 
\end{thm}

\begin{rem}\label{f-rem2.2}
If we assume that $L$ is ample, 
$l=1$, and $k$ is sufficiently large in Theorem \ref{f-thm2.1}, 
then we obtain that 
$$
H^i(X, \mathcal O_X(K_X+L))\hookrightarrow 
H^i(X, \mathcal O_X(K_X+(1+k)L))=0
$$ 
for every $i>0$ by Serre's vanishing theorem. 
Therefore, Theorem \ref{f-thm2.1} quickly recovers 
the Kodaira vanishing theorem for projective varieties 
(see Theorem \ref{f-thm2.3} below). 
\end{rem}

For the reader's convenience, 
we recall:  

\begin{thm}[Kodaira vanishing theorem for projective varieties]\label{f-thm2.3} 
Let $X$ be a smooth projective variety and let $L$ be 
an ample Cartier divisor on $X$. 
Then we have 
$$
H^i(X, \mathcal O_X(K_X+L))=0
$$ 
for every $i>0$. 
\end{thm}

We will give a proof of Theorem \ref{f-thm2.3} after 
we discuss $E_1$-degenerations of 
Hodge to de Rham type spectral sequences. 

Note that Theorem \ref{f-thm2.1} is obviously a generalization of 
Tankeev's pioneering result. 

\begin{thm}[{\cite[Proposition 1]{tankeev}}]\label{f-thm2.4}
Let $X$ be a smooth projective variety with 
$\dim X\geq 2$. 
Assume that the complete linear system $|L|$ has no 
base points and determines a morphism $\Phi_{|L|}: X\to Y$ onto a 
variety $Y$ with $\dim Y\geq 2$. 
Then 
$$
H^0(X, \mathcal O_X(K_X+2D))\to H^0(D, \mathcal O_D((K_X+2D)|_D))
$$ 
is surjective for almost all divisors $D\in |L|$. 
Equivalently, 
$$H^1(X, \mathcal O_X(K_X+D))\to H^1(X, \mathcal O_X(K_X+2D))$$ 
is injective for almost all divisors $D\in |L|$. 
\end{thm}

By Theorem \ref{f-thm2.1}, 
we can prove: 

\begin{thm}[{\cite[Theorem 2.1]{kollar-higher1}}]\label{f-thm2.5}
Let $X$ be a smooth projective variety, let $Y$ be 
an arbitrary projective variety, and let $f:X\to Y$ be a surjective 
morphism. 
Then we have the following properties. 
\begin{itemize}
\item[(i)] $R^if_*\mathcal O_X(K_X)$ is torsion-free for every $i$. 
\item[(ii)] Let $H$ be an ample Cartier divisor on $Y$, 
then $$
H^j(Y, \mathcal O_Y(H)\otimes R^if_*\mathcal O_X(K_X))=0
$$ 
for every $j>0$ and every $i$. 
\end{itemize}
\end{thm}

Theorem \ref{f-thm2.5} (i) and (ii) are called 
Koll\'ar's torsion-freeness and the Koll\'ar vanishing theorem respectively. 
We give a small remark on Theorem \ref{f-thm2.5}. 

\begin{rem}\label{f-rem2.6}
If $f=id _X: X\to X$ in Theorem \ref{f-thm2.5} (ii), 
then we have 
$H^i(X, \mathcal O_X(K_X+H))=0$ for every $i>0$ and 
every ample Cartier divisor $H$ on $X$. 
This is nothing but the Kodaira vanishing theorem for 
projective 
varieties (see Theorem \ref{f-thm2.3}). 
If $f$ is birational in Theorem \ref{f-thm2.5} (i), 
then $R^if_*\mathcal O_X(K_X)=0$ for 
every $i>0$ since 
$R^if_*\mathcal O_X(K_X)$ is a torsion sheaf for 
every $i>0$. 
This is the Grauert--Riemenschneider vanishing theorem 
for birational morphisms between projective varieties. 
\end{rem}

In \cite{kollar-higher1}, 
Koll\'ar proved Theorem \ref{f-thm2.1} and Theorem \ref{f-thm2.5} 
simultaneously. 
Therefore, 
the relationship between 
Theorem \ref{f-thm2.1} and Theorem \ref{f-thm2.5} is not clear in \cite{kollar-higher1}. 
Now it is well-known that 
Theorem \ref{f-thm2.1} and 
Theorem \ref{f-thm2.5} are equivalent by the works of 
Koll\'ar himself and Esnault--Viehweg (see, for example, 
\cite[Chapter 9]{kollar-book} and \cite{esnault-viehweg}). 
We note that Theorem \ref{f-thm2.1} follows from the $E_1$-degeneration of 
Hodge to de Rham spectral sequence. 

\begin{say}[$E_1$-degeneration of Hodge to de Rham spectral sequence]\label{f-say2.7}
Let $V$ be a smooth projective variety. 
Then the spectral sequence 
$$
E^{p, q}_1=H^q(V, \Omega^p_V)\Rightarrow 
H^{p+q}(V, \mathbb C)
$$ 
degenerates at $E_1$. 
This is a direct consequence of 
the Hodge decomposition for compact K\"ahler manifolds. 
\end{say}

Therefore, we can see that Theorem \ref{f-thm2.1} 
is a result of the theory of 
{\em{pure}} Hodge structures. 
Thus, it is natural to consider {\em{mixed}} 
generalizations of Theorem \ref{f-thm2.1}. 

We do not repeat the proof of 
Theorem \ref{f-thm2.1} depending on 
the $E_1$-degeneration of Hodge 
to de Rham spectral sequence in \ref{f-say2.7} here. 
For the details, see, for example, 
\cite[Chapter 9]{kollar-book} and \cite{esnault-viehweg}. 

\begin{say}\label{f-say2.8} 
We note Deligne's famous 
generalization of the $E_1$-degeneration in \ref{f-say2.7}. 
Let $V$  be a smooth projective variety and let $\Delta$ be a simple 
normal crossing divisor on $V$. Then 
the spectral sequence 
$$
E^{p, q}_1=H^q(V, \Omega^p_V(\log \Delta))\Rightarrow 
H^{p+q}(V\setminus \Delta, \mathbb C)
$$ 
degenerates at $E_1$ by 
Deligne's theory of mixed Hodge structures for 
smooth noncompact 
algebraic varieties (see \cite{deligne}). 
\end{say}

Unfortunately, 
the $E_1$-degeneration in \ref{f-say2.8} seems to produce no useful generalizations of 
Theorem \ref{f-thm2.1}. 
We think that 
the following $E_1$-degeneration is a correct ingredient for mixed generalizations of 
Theorem \ref{f-thm2.1}. 

\begin{say}\label{f-say2.9} 
Let $V$ and $\Delta$ be as in \ref{f-say2.8}. 
Then 
the spectral sequence 
$$
E^{p, q}_1=H^q(V, \Omega^p_V(\log \Delta)\otimes \mathcal O_V(-\Delta))\Rightarrow 
H^{p+q}_c(V\setminus \Delta, \mathbb C)
$$
degenerates at $E_1$. 
This 
is a consequence of mixed Hodge structures on cohomology with compact 
support $H^\bullet_c(V\setminus \Delta, \mathbb C)$. 
\end{say}

\begin{rem}\label{f-rem2.10}
In \ref{f-say2.9}, we see that 
$H^q(V, \Omega^p_V(\log \Delta)\otimes \mathcal O_V(-\Delta))$ 
is dual to $H^{n-q}(V, \Omega^{n-p}_V(\log \Delta))$ by Serre duality, 
where $n=\dim V$. 
Moreover, 
$H^{p+q}_c(V\setminus \Delta, 
\mathbb C)$ is dual to 
$H^{2n-(p+q)}(V\setminus \Delta, \mathbb C)$ by Poincar\'e 
duality. 
Therefore, we can check the $E_1$-degeneration in \ref{f-say2.9} 
by the $E_1$-degeneration in \ref{f-say2.8}. 
However, it is better to discuss mixed Hodge 
structures on cohomology with compact 
support in order to 
treat more general situations below (see Theorem \ref{f-thm2.12}, 
Theorem \ref{f-thm2.13}, Theorem \ref{f-thm3.6}, and so on). 
\end{rem}

We give a proof of the Kodaira vanishing theorem for 
projective varieties by using the 
$E_1$-degeneration in \ref{f-say2.9} in order to 
get the reader to grow more comfortable with 
the $E_1$-degeneration in \ref{f-say2.9}. 

\begin{proof}[Proof of Theorem \ref{f-thm2.3}] 
By the standard covering trick (see, 
for example, Step 1 in the proof of \cite[Theorem 2.2]{kollar-higher1}), 
we can reduce Theorem \ref{f-thm2.3} 
to the case when the complete linear system $|L|$ has no base points. 
So, we assume that $|L|$ has no base points for simplicity. 
We take a smooth member $D$ of $|L|$ by Bertini's theorem. 
We put $\iota: X\setminus D\hookrightarrow X$. 
By the $E_1$-degeneration of 
$$
E^{p, q}_1=H^q(X, \Omega^p_X(\log D)\otimes \mathcal O_X(-D))
\Rightarrow H^{p+q}_c(X\setminus D, \mathbb C),  
$$ 
we obtain that the natural map 
$$
\pi: H^j(X, \iota_!\mathbb C_{X\setminus D})
\to H^j(X, \mathcal O_X(-D))
$$ 
induced by $\iota_!\mathbb C_{X\setminus D}\subset 
\mathcal O_X(-D)$ is surjective 
for every $j$. 
Since $$\iota_!\mathbb C_{X\setminus D}\subset \mathcal O_X(-mD)\subset 
\mathcal O_X(-D)$$ for every $m\geq 1$, 
we obtain that 
$$
\pi:H^j(X, \iota_! \mathbb C_{X\setminus D})\to 
H^j(X, \mathcal O_X(-mD))\overset{p}\to 
H^j(X, \mathcal O_X(-D))
$$ and 
that $p$ is surjective for every $j$. 
Note that $H^j(X, \mathcal O_X(-mD))=0$ for 
$j<\dim X$ and for 
$m\gg 0$ by the Serre vanishing theorem. 
Thus we obtain that 
$H^j(X, \mathcal O_X(-D))=0$ for 
$j<\dim X$. 
By Serre duality, 
we have 
$H^i(X, \mathcal O_X(K_X+D))=0$ for 
every $i>0$. 
\end{proof}

We next give a remark on \cite{esnault-viehweg}. 

\begin{rem}\label{f-rem2.11}Let $V$ be a smooth 
projective variety and let $A+B$ be a simple normal crossing divisor 
on $V$ such that $A$ and $B$ have no common irreducible components. 
In \cite{esnault-viehweg}, 
Esnault--Viehweg discussed the $E_1$-degeneration 
of 
\begin{align*}
E^{p, q}_1&=H^q(V, \Omega^p_V(\log (A+B))\otimes 
\mathcal O_V(-B)) \\ &\Rightarrow 
\mathbb H^{p+q}(V, \Omega^\bullet _V(\log (A+B))\otimes 
\mathcal O_V(-B)) 
\end{align*} (see also \cite{deligne-illusie}). 
This $E_1$-degeneration contains 
the $E_1$-degenerations in \ref{f-say2.8} and 
in \ref{f-say2.9} as special cases. 
However, 
they did not pursue geometric applications of the $E_1$-degeneration 
in \ref{f-say2.9}, that is, in the case when $A=0$. 
\end{rem}

By using the $E_1$-degeneration in \ref{f-say2.9} and 
some more general $E_1$-degenerations 
arising from mixed Hodge structures on cohomology with 
compact support, we can obtain various generalizations of 
Theorem \ref{f-thm2.1} and Theorem \ref{f-thm2.5}. 
We write the following useful generalizations 
without explaining the precise definitions and 
the notation here (see \ref{f-say6.6}, 
\ref{f-say6.7}, \ref{f-say6.8}, \ref{f-say6.9} in Section \ref{f-sec6}). 

\begin{thm}[Injectivity theorem for simple normal crossing pairs]\label{f-thm2.12} 
Let $(X, \Delta)$ be a simple normal crossing 
pair such that $\Delta$ is an 
$\mathbb R$-divisor on $X$ whose coefficients are in 
$[0, 1]$, and let $\pi:X\to V$ be a proper 
morphism between schemes. 
Let $L$ be a Cartier 
divisor on $X$ and let $D$ be an 
effective Cartier divisor that is permissible 
with 
respect to $(X, \Delta)$. 
Assume the following conditions. 
\begin{itemize}
\item[(i)] $L\sim _{\mathbb R, \pi}K_X+\Delta+H$, 
\item[(ii)] $H$ is a $\pi$-semiample 
$\mathbb R$-divisor, and 
\item[(iii)] $tH\sim _{\mathbb R, \pi} D+D'$ for some 
positive real number $t$, where 
$D'$ is an effective $\mathbb R$-Cartier 
$\mathbb R$-divisor that is permissible with respect to $(X, \Delta)$. 
\end{itemize}
Then the homomorphisms 
$$
R^q\pi_*\mathcal O_X(L)\to R^q\pi_*\mathcal O_X(L+D), 
$$ 
which are induced by the natural inclusion 
$\mathcal O_X\hookrightarrow \mathcal O_X(D)$, 
are injective for all $q$. 
\end{thm}

Theorem \ref{f-thm2.12} is a generalization of 
Theorem \ref{f-thm2.1}. 

\begin{thm}\label{f-thm2.13} 
Let $f:(Y, \Delta)\to X$ be a proper morphism from an embedded simple 
normal crossing pair $(Y, \Delta)$ to 
a scheme $X$ such that 
$\Delta$ is an $\mathbb R$-divisor whose coefficients are in $[0, 1]$. 
Let $L$ be a Cartier divisor on $Y$ and let $q$ be an arbitrary nonnegative 
integer. Then we have the following properties. 
\begin{itemize}
\item[(i)] 
Assume that $L-(K_Y+\Delta)$ is $f$-semi-ample. Then  
every associated prime of 
$R^qf_*\mathcal O_Y(L)$ is the generic point 
of the $f$-image of 
some stratum of $(Y, \Delta)$. 
\item[(ii)] Let $\pi:X\to V$ 
be a proper morphism between schemes. 
Assume that $$f^*H\sim _{\mathbb R}L-(K_Y+\Delta), $$ 
where $H$ is nef and log big over $V$ with 
respect to $f:(Y, \Delta)\to X$. 
Then we have $$R^p\pi_*R^qf_*\mathcal O_Y(L)=0$$ for every $p>0$. 
\end{itemize}
\end{thm}

Theorem \ref{f-thm2.13} (i) and (ii) are generalizations of Theorem \ref{f-thm2.5} 
(i) and (ii) respectively. 
For the details, see, for example, \cite[Sections 5 and 
6]{fujino-fundamental}, \cite[Theorem 1.1]
{fujino-vanishing}, \cite[Theorem 1.1]{fujino-injectivity}, 
and \cite[Chapter 5]{fujino-foundation}. 
Note that Theorem \ref{f-thm2.12} and Theorem \ref{f-thm2.13} 
have already played crucial roles 
in the proof of the fundamental theorems for 
log canonical pairs and semi-log canonical pairs 
(see, for example, \cite{fujino-fundamental}, 
\cite{fujino-slc}, and \cite{fujino-foundation}). 

Anyway, the formulation of Theorem \ref{f-thm2.12} and 
Theorem \ref{f-thm2.13} is natural and 
useful from the minimal model theoretic viewpoint, although 
it may look unduly technical and artificial. 

\begin{rem}\label{f-rem2.14}
Let $V$ and $\Delta$ be as in \ref{f-say2.8}. 
In the traditional framework, 
$\mathcal O_V(K_V+\Delta)$ was recognized to 
be $\det \Omega^1_V(\log \Delta)$. 
On the other hand, 
in our new framework for vanishing theorems, 
we see $\mathcal O_V(K_V+\Delta)$ as $$\mathcal Hom _{\mathcal O_V}
(\mathcal O_V(-\Delta), \mathcal O_V(K_V))$$ and $\mathcal O_V(-\Delta)$ as the 
$0$th term of $\Omega^\bullet _V(\log \Delta)\otimes 
\mathcal O_V(-\Delta)$. 
\end{rem}

I think that it is not so easy to 
understand the statements of Theorem \ref{f-thm2.12} and Theorem \ref{f-thm2.13}. 
So we give a very special 
case of Theorem \ref{f-thm2.12} to 
clarify the main difference between Theorem \ref{f-thm2.1} and Theorem \ref{f-thm2.12}. 

\begin{thm}\label{f-thm2.15} 
Let $X$ be a smooth projective variety and 
let $\Delta$ be a simple 
normal crossing divisor on $X$. 
Let $L$ be a semiample  
Cartier divisor on $X$ and let 
$D$ be a member of $|kL|$ for some positive integer 
$k$ such that 
$D$ contains no strata of $\Delta$. 
Then the homomorphism 
$$
H^i(X, \mathcal O_X(K_X+\Delta+lL))\to 
H^i(X, \mathcal O_X(K_X+\Delta+(l+k)L))
$$ 
induced by the natural inclusion 
$\mathcal O_X\hookrightarrow \mathcal O_X(D)$ is injective for 
every positive integer $l$ and every $i$. 
\end{thm}

If $\Delta=0$ in the above, 
then Theorem \ref{f-thm2.15} is nothing but Koll\'ar's original 
injectivity theorem (Theorem \ref{f-thm2.1}). 

\begin{rem}\label{f-rem2.16}
Let $\Delta$ be a simple normal crossing divisor on a smooth 
variety $X$. 
Let $\Delta=\sum _{i\in I}\Delta_i$ be the irreducible decomposition 
of $\Delta$. 
Then a closed subset $W$ of $X$ is 
called a stratum of $\Delta$ if 
$W$ is an irreducible component of 
$\Delta_{i_1}\cap \cdots \cap \Delta_{i_k}$ for 
some $\{i_1, \cdots, i_k\}\subset I$. 
\end{rem}

\begin{rem}\label{f-rem2.17} 
Let $\Delta$ be a simple normal crossing divisor on a smooth 
variety $V$. 
Then $W$ is a stratum of $\Delta$ if and only if 
$W$ is a log canonical center of $(V, \Delta)$ 
(see \ref{f-say6.5} in Section \ref{f-sec6}).  
\end{rem}

We have discussed the Hodge theoretic aspect of 
Kodaira-type vanishing theorems. 
For the details and various related topics, see 
\cite{esnault-viehweg},\cite[Part III]{kollar-book}, \cite{fujino-foundation}, 
and references therein. 

\subsection{Complex analytic setting}\label{f-subsec2.1}
After Koll\'ar obtained Theorem \ref{f-thm2.1}, 
Enoki (see \cite[Theorem 0.2]{enoki}) proved: 

\begin{thm}[Enoki's injectivity theorem]\label{f-thm2.18}
Let $X$ be a compact K\"ahler manifold and 
let $\mathcal L$ be a semipositive line bundle on $X$. 
Then, for any nonzero holomorphic section $s$ of $\mathcal L^{\otimes k}$ 
with 
some positive integer $k$,  
the multiplication homomorphism 
$$
\times s: H^i(X, \omega_X\otimes \mathcal L^{\otimes l})\longrightarrow H^i(X, \omega_X
\otimes \mathcal L^{\otimes (l+k)}), 
$$ 
induced by $\otimes s$, 
is injective for every $i$ and every positive integer $l$. 
\end{thm} 

\begin{rem}\label{f-rem2.19}
Let $\mathcal L$ be a holomorphic 
line bundle on a compact K\"ahler 
manifold $X$. 
We say that 
$\mathcal L$ is semipositive 
if there exists a smooth hermitian metric 
$h$ on $\mathcal L$ such that $\sqrt{-1}\Theta_h(\mathcal L)$ is 
a semipositive 
$(1, 1)$-form on $X$, 
where $\Theta_h(\mathcal L)=D^2_{(\mathcal L, h)}$ is the 
curvature form and $D_{(\mathcal L, h)}$ is the Chern connection 
of $(\mathcal L, h)$. 
\end{rem}

\begin{rem}\label{f-rem2.20} 
Let $X$ be a smooth projective variety and let $\mathcal L$ be 
a line bundle on $X$. 
If $\mathcal L$ is semiample, that is, 
$|\mathcal L^{\otimes k}|$ has no base points for some 
positive integer $k$, 
then $\mathcal L$ is semipositive in the sense of Remark \ref{f-rem2.19}.  
\end{rem}

Enoki's proof in \cite{enoki} 
is arguably simpler than the proof of Theorem \ref{f-thm2.1} 
based on Hodge theory. 
It only uses the standard results in 
the theory of harmonic forms on compact 
K\"ahler manifolds. Let us see the ideas of Enoki's 
proof of Theorem \ref{f-thm2.18}. 

\begin{proof}[Idea of Proof of Theorem \ref{f-thm2.18}] 
We put $n=\dim X$. 
Let $\mathcal H^{n, i}(X, \mathcal L^{\otimes l})$ 
(resp.~$\mathcal H^{n, i}(X, \mathcal L^{\otimes (l+k)})$ 
be the space of $\mathcal L^{\otimes l}$-valued 
(resp.~$\mathcal L^{\otimes (l+k)}$-valued) harmonic $(n, i)$-forms 
on $X$. 
By using the Nakano identity and the semipositivity of 
$\mathcal L$, 
we can easily check that 
$s\otimes \varphi$ is harmonic 
for every $\varphi\in \mathcal H^{n, i}(X, \mathcal L^{\otimes l})$. 
Therefore, 
$$
\times s: H^i(X, \omega_X\otimes \mathcal L^{\otimes l})\longrightarrow H^i(X, \omega_X
\otimes \mathcal L^{\otimes (l+k)}), 
$$ 
is nothing but 
$\otimes s: \mathcal H^{n, i}(X, \mathcal L^{\otimes l}) \to 
\mathcal H^{n, i}(X, \mathcal L^{\otimes (l+k)}): \varphi\mapsto 
s\otimes  \varphi$, 
which is obviously injective. 
\end{proof}

We note that Theorem \ref{f-thm2.18} is better than 
Theorem \ref{f-thm2.1} by Remark \ref{f-rem2.20}. 
Unfortunately, I do not know how to generalize Enoki's theorem 
appropriately for various geometric applications. 
Although I obtained some generalizations of Theorem \ref{f-thm2.18} 
and their applications in \cite{fujino-trans} and 
\cite{fujino-trans2}, 
they are not so useful in the 
minimal model program compared with 
Theorem \ref{f-thm2.12} and Theorem \ref{f-thm2.13}. 
Related to Theorem \ref{f-thm2.15}, we have: 

\begin{conj}\label{f-conj2.21}
Let $X$ be a compact K\"ahler manifold and 
let $\Delta$ be a simple normal crossing divisor on $X$. 
Let $\mathcal L$ be a semipositive line bundle on $X$ and let $s$ be a 
nonzero holomorphic section of 
$\mathcal L^{\otimes k}$ on $X$ for some positive integer $k$. 
Assume that $(s=0)$ contains no strata of $\Delta$. 
Then the multiplication homomorphism 
$$
\times s: H^i(X, \omega_X\otimes \mathcal O_X(\Delta)\otimes 
\mathcal L^{\otimes l})\to H^i(X, \omega_X\otimes \mathcal O_X(\Delta)\otimes 
\mathcal L^{\otimes (l+k)}), 
$$ 
induced by $\otimes s$, is injective for every positive integer $l$ and every 
$i$. 
\end{conj}

I do not know the precise relationship between 
the injectivity theorems of Koll\'ar and Enoki. Thus I pose: 

\begin{problem}\label{f-prob2.22}
Clarify the relationship between Koll\'ar's injectivity 
theorem (Theorem \ref{f-thm2.1}) and 
Enoki's injectivity theorem (Theorem \ref{f-thm2.18}). 
\end{problem}

For almost all geometric applications, we use Theorem \ref{f-thm2.5} (ii) 
for $i=0$. Theorem \ref{f-thm2.5} (ii) for $i=0$ is sufficient 
for Viehweg's theory of 
weak positivity (see \cite{viehweg1}, \cite{viehweg2}, 
and \cite{fujino-subadditivity}). See also Subsection \ref{f-subsec4.2} below. 
Note that Theorem \ref{f-thm2.5} (ii) for $i=0$ is a special case of 
Ohsawa's vanishing theorem. 

\begin{thm}[{\cite[Theorem 3.1]{ohsawa}}]\label{f-thm2.23} 
Let $X$ be a compact K\"ahler manifold, let 
$f:X\to Y$ be a holomorphic 
map to an analytic space $Y$ with 
a K\"ahler form $\sigma$, and let $(E, h)$ be a holomorphic 
vector bundle on $X$ 
with a smooth hermitian metric $h$. 
Assume that $\sqrt{-1}\Theta_h(E)\geq _{\mathrm{Nak}}
\mathrm{Id}_E\otimes f^*\sigma$, that is, 
$\sqrt{-1}\Theta _h (E)-\mathrm{Id}_E\otimes f^*\sigma$ is 
semipositive in the sense of Nakano, 
where $\Theta_h(E)$ is the curvature form of $(E, h)$. 
Then 
$$
H^j(Y, f_*(\omega_X\otimes E))=0
$$ 
for every $j>0$. 
\end{thm}

For the proof of Theorem \ref{f-thm2.23}, see \cite{ohsawa}. 
I am not so familiar with 
Theorem \ref{f-thm2.23} and do not know 
if the formulation of Theorem \ref{f-thm2.23} 
is natural or not. 

\begin{rem}\label{f-rem2.24} 
For the details of $\sigma$ and $f^*\sigma$ in 
Theorem \ref{f-thm2.23}, see \cite[\S 3]{ohsawa}. 
Note that $Y$ is permitted to have singularities. 
\end{rem}

By comparing Theorem \ref{f-thm2.23} 
with Theorem \ref{f-thm2.5} (ii), it is natural to consider: 

\begin{conj}\label{f-conj2.25}
With the same assumptions as in Theorem \ref{f-thm2.23}, 
we have 
$$
H^j(Y, R^if_*(\omega_X\otimes E))=0
$$ 
for every $i$ and every positive integer $j$. 
\end{conj}

We close this section with: 

\begin{problem}\label{f-prob2.26}
Clarify the relationship 
between Koll\'ar's vanishing theorem 
(Theorem \ref{f-thm2.5} (ii)) and Ohsawa's vanishing 
theorem (Theorem \ref{f-thm2.23}). 
\end{problem}

For Enoki-type injectivity theorems, 
see, \cite{enoki}, 
\cite{takegoshi}, 
\cite{ohsawa-injectivity}, 
\cite{fujino-trans}, \cite{fujino-trans2}, \cite{matsumura1}, 
\cite{matsumura2}, \cite{matsumura3}, \cite{matsumura4}, 
\cite{matsumura5}, \cite{gongyo-matsumura}, etc.  

\begin{rem}[Added in September 2016]\label{f-rem2.27}
After I wrote this paper, there are some developments. 
In \cite{matsumura8}, Shin-ichi Matsumura proved Conjecture 
\ref{f-conj2.21} under the extra assumption that 
$\Delta$ is smooth (see \cite[Theorem 1.3 and Corollary 
1.4]{matsumura8}). 
He also proved Conjecture \ref{f-conj2.25} completely in 
\cite{matsumura6}. 
For the precise statement, see \cite[Theorem 1.3]{matsumura6}. 
In \cite{fujino-matsumura}, 
he and I obtained a generalization of Enoki's injectivity theorem, 
which can be seen as a generalization of Nedel's vanishing 
theorem (see \cite[Theorem A]{fujino-matsumura}). 
For some further developments, see \cite{fujino-kollar-nadel} and 
\cite{matsumura7}. 
\end{rem}

\section{On local freeness and 
semipositivity theorems}\label{f-sec3}

Let us start with Fujita's semipositivity theorem in \cite{fujita-kahler}. 

\begin{thm}[{\cite[(0.6) Main Theorem]{fujita-kahler}}]\label{f-thm3.1} 
Let $f:M\to C$ be a surjective morphism from a compact K\"ahler manifold 
onto a smooth projective 
curve $C$ with 
connected fibers. 
Then $f_*\omega_{M/C}$ is nef. 
\end{thm}

Before we go further, 
let us recall the definition of nef locally free sheaves. 
 
\begin{defn}[Nef locally free sheaves]\label{f-def3.2} 
Let $\mathcal E$ be a locally free sheaf of finite rank on a 
complete algebraic variety $V$. 
Then $\mathcal E$ is called nef if $\mathcal E=0$ or  
$\mathcal O_{\mathbb P_V(\mathcal E)}(1)$ is 
nef on $\mathbb P_V(\mathcal E)$, 
that is, $\mathcal O_{\mathbb P_V(\mathcal E)}(1)\cdot C\geq 0$ for 
every curve $C$ on $\mathbb P_V(\mathcal E)$. 
A nef locally free sheaf $\mathcal E$ 
was originally called a (numerically) semipositive 
locally free sheaf in the literature. 
\end{defn}

\begin{rem}\label{f-rem3.3}
Assume that $X$ is a smooth projective variety for simplicity. 
Let $\mathcal L$ be a line bundle on $X$. 
Then $\mathcal L$ is nef in the sense of 
Definition \ref{f-def3.2} if and only if 
$\mathcal L$ is nef in the usual sense. 
If $\mathcal L$ is semipositive in the sense of 
Remark \ref{f-rem2.19}, then $\mathcal L$ is nef. 
However, a nef line bundle $\mathcal L$ is not necessarily semipositive 
in the sense of Remark \ref{f-rem2.19}. 
\end{rem}

\begin{rem}\label{f-rem3.4}
Note that $f$ is not necessarily smooth in Theorem \ref{f-thm3.1}. 
If $f$ is smooth 
in Theorem \ref{f-thm3.1}, 
then the nefness of $f_*\omega_{M/C}$ follows directly from Griffiths's 
calculations of connections and curvatures in \cite{griffiths}. 
\end{rem}

Although Fujita's theorem was inspired by Griffiths's paper \cite{griffiths} 
(see Remark \ref{f-rem3.4}), 
Fujita's original proof of Theorem \ref{f-thm3.1} in \cite{fujita-kahler} is not 
so Hodge theoretic. 
In the introduction of \cite{fujita-kahler}, 
Fujita wrote: 
\begin{quote}
The method looks rather elementary and 
purely computational, but 
it depends deeply (often implicitly) on the theory 
of variation of Hodge structures. 
\end{quote}

Professor Steven Zucker informed me that 
he read Fujita's article \cite{fujita-kahler} when it appeared in 
1978 and reproved Fujita's theorem from rather basic 
Hodge theory that appeals to 
Steenbrink's work \cite{steenbrink}. 
It is clear that he had already been very familiar with Schmid's 
result (see \cite{schmid}) on asymptotic behavior of Hodge metrics (see 
\cite{zucker1} and \cite{zucker2}). 
I think that he could write \cite{zucker3} without 
any difficulties. 
It seems that he is the first one who directly applies Hodge theory to 
obtain semipositivity results like Theorem \ref{f-thm3.1}, that is, 
semipositivity results for nonsmooth morphisms. 

Independently, 
Kawamata obtained the following semipositivity theorem 
in \cite{kawamata-chara} by using Schmid's paper \cite{schmid}. 
His result is: 

\begin{thm}[{\cite[Theorem 5]{kawamata-chara}}]\label{f-thm3.5} 
Let $f:X\to Y$ be a surjective morphism between smooth projective 
varieties with connected fibers which satisfies the following conditions: 
\begin{itemize}
\item[(i)] There is a Zariski open dense subset $Y_0$ of $Y$ such that 
$D=Y\setminus Y_0$ is a simple normal crossing divisor on $Y$. 
\item[(ii)] Put $X_0=f^{-1}(Y_0)$ and $f_0=f|_{X_0}$. Then 
$f_0$ is smooth. 
\item[(iii)] The local monodromies 
of $R^nf_{0*}\mathbb C_{X_0}$ around $D$ are unipotent, 
where $n=\dim X-\dim Y$. 
\end{itemize} 
Then $f_*\omega_{X/Y}$ is a locally free sheaf and nef. 
\end{thm}

However, the proof of Theorem \ref{f-thm3.5} in 
\cite{kawamata-chara} seems to be insufficient when 
$\dim Y\geq 2$. We do not repeat the comments 
on the troubles in \cite{kawamata-chara} here. 
For the details, see the comments 
in \cite[4.6.~Remarks]{ffs}. 
Fortunately, we have some 
generalizations of Theorem \ref{f-thm3.5} in \cite{fujino-higher}, 
\cite{fujino-fujisawa}, and 
\cite{ffs} (see, for example, Theorem \ref{f-thm3.6} below). 
The proofs in \cite{fujino-fujisawa} and \cite{ffs} 
are independent of Kawamata's arguments in \cite{kawamata-chara}. 
Note that our arguments in \cite{fujino-fujisawa} and 
\cite{ffs} need some results on Hodge theory obtained 
after the publication of Kawamata's paper \cite{kawamata-chara} 
(see, for example, \cite{cattani-kaplan}, and \cite{cks}). 
Kawamata could and did use 
only \cite{deligne}, \cite{griffiths}, and 
\cite{schmid} on Hodge theory when he wrote \cite{kawamata-chara}. 
Although I sometimes called Theorem \ref{f-thm3.5} the Fujita--Kawamata 
semipositivity theorem (see, for example, \cite{fujino-fujisawa}), 
it is probably not accurate. 
It is more appropriate to call it the Fujita--Zucker--Kawamata 
semipositivity theorem. 
I apologize for suppressing Zucker's contribution \cite{zucker3}. 

Theorem \ref{f-thm3.5} follows from the theory of 
polarizable 
variation of {\em{pure}} Hodge structure. 
It is natural to consider {\em{mixed}} generalizations of Theorem \ref{f-thm3.5}. 
We have already seen that 
mixed Hodge structures on cohomology with compact support 
are very useful (see Section \ref{f-sec2}). 
So, we consider their variations and prove 
some powerful generalizations of Theorem \ref{f-thm3.5}, 
which depend on the theory of gradedly polarizable 
admissible variation of {\em{mixed}} Hodge structure (see, 
for example, \cite{steenbrink-zucker}, and \cite{kashiwara}). 
We have: 

\begin{thm}[Semipositivity theorem]\label{f-thm3.6}
Let $(X, D)$ be a simple normal crossing pair such that 
$D$ is reduced and let $f:X\to Y$ be 
a projective surjective morphism onto a smooth 
complete algebraic variety $Y$.
Assume that every stratum of $(X,D)$ is dominant onto $Y$. 
Let $\Sigma$ be a simple normal crossing divisor on $Y$ such that 
every stratum of $(X, D)$ is smooth over $Y^*=Y\setminus \Sigma$. 
Then $R^pf_*\omega_{X/Y}(D)$ is locally free for every $p$. 
We put $X^*=f^{-1}(Y^*)$,
$D^*=D|_{X^*}$, and $d=\dim X-\dim Y$. 
We further assume that 
all the local monodromies on 
$R^{d-i}(f|_{X^*\setminus D^*})_!\mathbb Q_{X^* \setminus D^*}$ 
around $\Sigma$ 
are unipotent. 
Then we obtain that
$R^if_*\omega_{X/Y}(D)$ is 
a nef locally free sheaf on $Y$. 
\end{thm}

For the definitions and the notation used in Theorem \ref{f-thm3.6}, 
see \ref{f-say6.6} and \ref{f-say6.7} in Section \ref{f-sec6}. 
Theorem \ref{f-thm3.6} was first obtained in \cite{fujino-fujisawa}. 
Then, we gave an alternative proof of 
Theorem \ref{f-thm3.6} based on Saito's theory of mixed Hodge 
modules (see \cite{saito1}, \cite{saito2}, and \cite{saito3}) in \cite{ffs}. 
As an application of Theorem \ref{f-thm3.6}, 
we establish the projectivity of various moduli spaces 
(for the details, see \cite{fujino-slc}, 
\cite{fujino-semipositivity}, \cite{kovacs-p}, etc.). 
For a new approach, see \cite{fujino-vani-semi}.  
In the introduction of \cite{fujita-kahler}, Fujita wrote: 
\begin{quote}
Perhaps our result is closely related with 
the problem about the (quasi-)projectivity 
of moduli spaces. 
Of course, however, the relation will not be simple. 
\end{quote}

Now we know that generalizations of Fujita's semipositivity theorem 
(see Theorem \ref{f-thm3.1} and Theorem \ref{f-thm3.6}) with Viehweg's 
covering arguments (see \cite{viehweg1} and \cite{viehweg2}) 
are useful for the projectivity 
of coarse moduli spaces of stable (log-)varieties 
(see, for example, \cite{kollar-projectivity}, \cite{fujino-semipositivity}, 
and \cite{kovacs-p}). 

Anyway, by Theorem \ref{f-thm3.5}, 
we have: 
\begin{thm}[Fujita, Zucker, Kawamata, $\cdots$]
\label{f-thm3.7} Let $f:X\to Y$ be a surjective morphism between 
smooth projective varieties with connected fibers. 
Then there exists a generically finite morphism $\tau:Y'\to Y$ from 
a smooth projective variety $Y'$ with 
the following property. 
Let $X'$ be any resolution of the main component of $X\times _Y Y'$. 
Then $f'_*\omega_{X'/Y'}$ is a nef locally free sheaf, 
where $f'$ is the composite $X'\to X\times _Y Y'\to Y'$. 
\end{thm}

Theorem \ref{f-thm3.7} has already played crucial roles 
in the study of higher-dimensional algebraic varieties. 
For some geometric applications, we have to 
treat $f_*\omega^{\otimes m}_{X/Y}$ or 
$f'_*\omega^{\otimes m}_{X'/Y'}$ with $m\geq 2$ (see Section \ref{f-sec4}). 
Thus we have: 

\begin{conj}[Semipositivity of direct images of relative 
pluricanonical 
bundles]\label{f-conj3.8}
Let $f:X\to Y$ be a surjective morphism between 
smooth projective varieties with connected fibers. 
Then there exists a generically finite morphism 
$\tau:Y'\to Y$ from a smooth projective variety $Y'$ 
with the following property. 
Let $X'$ be any resolution of the main component of $X\times _Y Y'$ 
sitting in the following commutative diagram: 
$$
\xymatrix{X' \ar[r]\ar[d]_{f'}& X\ar[d]^f\\
Y'\ar[r]_{\tau} &Y. 
}
$$ 
Then $f'_*\omega^{\otimes m}_{X'/Y'}$ is a 
nef locally free sheaf for every positive integer $m$. 
\end{conj}

Note that the local freeness of $f'_*\omega^{\otimes m}_{X'/Y'}$ 
for $m\geq 2$ in Conjecture \ref{f-conj3.8} 
is highly nontrivial even when $f'$ is a smooth projective 
morphism. 
The following theorem by Siu (see \cite{siu}) 
is nontrivial for $m\geq 2$ and 
can be proved only by using $L^2$ methods. 
For a simpler proof, see \cite{paun}. 

\begin{thm}[Siu]\label{f-thm3.9}
Let $f:X\to Y$ be a smooth projective morphism 
between smooth quasiprojective varieties with connected fibers. 
Then $f_*\omega^{\otimes m}_{X/Y}$ is locally free for every nonnegative integer $m$. 
\end{thm}

Theorem \ref{f-thm3.9} is a clever application of the Ohsawa--Takegoshi 
$L^2$ extension 
theorem. 
We have no Hodge theoretic proofs of Theorem \ref{f-thm3.9}. 
Therefore, we ask: 

\begin{problem}\label{f-prob3.10}
Find a Hodge theoretic proof or an algebraic proof of Theorem \ref{f-thm3.9}. 
\end{problem}

We note: 

\begin{rem}\label{f-rem3.11}
If $Y$ is projective in Theorem \ref{f-thm3.9}, 
then $f_*\omega^{\otimes m}_{X/Y}$ is nef 
for every positive integer $m$ 
by Theorem \ref{f-thm3.17} below. 
Therefore, Conjecture \ref{f-conj3.8} holds true 
when $f:X\to Y$ is smooth. 
\end{rem}

We recommend the reader to see \cite{fujino-fujisawa} and \cite{ffs} for 
the Hodge theoretic aspect of semipositivity theorems discussed in this section. 
Note that the style of \cite{fujino-fujisawa} is the same as 
my other papers, whereas, \cite{ffs} is written in the language of Saito's 
theory of mixed Hodge modules. 

\subsection{New semipositivity theorems using MMP}\label{f-subsec3.1} 

In this subsection, we discuss new semipositivity theorems 
obtained through the use of 
the minimal model program, following \cite{fujino-direct} and 
\cite{fujino-direct-corr}. 

Let us start with the definition of (good) minimal models. 
We recommend the reader to see \ref{f-say6.2} and \ref{f-say6.5} 
in Section \ref{f-sec6} if he is not familiar with 
the minimal model program. 

\begin{defn}[Good minimal models]\label{f-def3.12}
Let $f:X\to Y$ be a projective morphism between normal 
quasiprojective 
varieties. 
Let $\Delta$ be an effective $\mathbb Q$-divisor on $X$ such that 
$(X, \Delta)$ is kawamata log terminal. 
A pair $(X', \Delta')$ sitting in a diagram 
$$
\xymatrix{
X\ar[dr]_{f}\ar@{-->}[rr]^\phi&& X'\ar[dl]^{f'}\\
&Y&
}
$$
is called a minimal model of $(X, \Delta)$ over $Y$ if 
\begin{itemize}
\item[(i)] $X'$ is $\mathbb Q$-factorial, 
\item[(ii)] $f'$ is projective, 
\item[(iii)] $\phi$ is birational and $\phi^{-1}$ has no 
exceptional divisors, 
\item[(iv)] $\phi_*\Delta=\Delta'$, 
\item[(v)] $K_{X'}+\Delta'$ is $f'$-nef, and 
\item[(vi)] $a(E, X, \Delta)<a(E, X', \Delta')$ for every $\phi$-exceptional 
divisor $E$ contained in $X$. 
\end{itemize}
Furthermore, if $K_{X'}+\Delta'$ is $f'$-semiample, 
then $(X', \Delta')$ is called a {\em{good}} minimal model 
of $(X, \Delta)$ over $Y$. 
When $Y$ is a point, we usually omit \lq\lq over $Y$\rq\rq \ in the above definitions;  
we sometimes simply say that $(X', \Delta')$ is a relative 
$($good$)$ minimal model of $(X, \Delta)$. 
\end{defn}

We also need the notion of weakly semistable morphisms 
introduced by Abramovich--Karu (see \cite{abramovich-karu}). 

\begin{defn}[Weakly semistable morphisms]\label{f-def3.13}
Let $f:X\to Y$ be a projective surjective morphism between quasiprojective 
varieties. 
Then $f:X\to Y$ is called weakly semistable if 
\begin{itemize}
\item[(i)] the varieties $X$ and $Y$ admit toroidal structures $(U_X\subset 
X)$ and $(U_Y\subset Y)$ with $U_X=f^{-1}(U_Y)$, 
\item[(ii)] with this structure, the morphism $f$ is toroidal, 
\item[(iii)] the morphism $f$ is equidimensional, 
\item[(iv)] all the fibers of the morphism $f$ are reduced, and 
\item[(v)] $Y$ is smooth. 
\end{itemize} 
It then follows that $X$ has only rational Gorenstein singularities 
(see \cite[Lemma 6.1]{abramovich-karu}). 
Both $(U_X\subset X)$ and $(U_Y\subset Y)$ are toroidal embeddings without 
self-intersection in the sense of \cite[Chapter II, \S1]{kkms}. 
For the details, see \cite{abramovich-karu}. 
\end{defn}

We propose the following conjecture. 

\begin{conj}\label{f-conj3.14} 
Let $f:X\to Y$ be a weakly semistable morphism 
with connected fibers. 
Then $f_*\omega^{\otimes m}_{X/Y}$ is locally free 
for every $m\geq 1$.
\end{conj}

By the argument in \cite[Section 4]{fujino-direct} (see also \cite{fujino-direct-corr}), 
we have: 

\begin{thm}[Local freeness]\label{f-thm3.15}
Let $f:X\to Y$ be a weakly semistable morphism with connected fibers. 
Assume that the geometric generic fiber $X_{\overline \eta}$ of $f:X\to Y$ 
has a good minimal model. 
Then $f_*\omega^{\otimes m}_{X/Y}$ is locally free 
for every $m\geq 1$. 
\end{thm}

\begin{proof}[Sketch of Proof of Theorem \ref{f-thm3.15}]
Let us consider the following commutative diagram: 
$$
\xymatrix{
X \ar[dr]_{f}\ar@{-->}[rr]^{\phi}& & \widetilde X\ar[dl]^{\widetilde f} \\
& Y &
}
$$
where $\widetilde f:\widetilde X\to Y$ is a relative 
good minimal model of $f: X\to Y$. 
We can always construct a relative good minimal model 
by the assumption that the geometric generic fiber of $f$ has a 
good minimal model. 
Then we have 
\begin{equation}\tag{$\heartsuit$}\label{isom-spade}
f_*\omega^{\otimes m}_{X/Y}
\simeq \widetilde f_*\mathcal O_{\widetilde X}(mK_{\widetilde X/Y})
\end{equation} 
for every positive integer $m$. 
We note that 
$X$ has only rational Gorenstein singularities. 

The following lemma due to Nakayama is a variant of 
Koll\'ar's torsion-freeness:~Theorem \ref{f-thm2.5} (i). 
It is a key ingredient of the proof of Theorem \ref{f-thm3.15}. 

\begin{lem}[{\cite[Corollary 3]{nakayama1}}]\label{f-lem3.16}
Let $g:V\to C$ be a projective surjective morphism 
from a normal quasiprojective 
variety $V$ to a smooth 
quasiprojective curve $C$. 
Assume that $V$ has only canonical singularities and 
that $K_V$ is $g$-semiample. 
Then $R^ig_*\mathcal O_V(mK_V)$ is locally free for every $i$ and 
every positive integer $m$. 
\end{lem}
By the above isomorphism \eqref{isom-spade}, it is sufficient to prove the local freeness 
of $\widetilde f_*\mathcal O_{\widetilde X}(mK_{\widetilde X/Y})$. 
Let $P$ be an arbitrary closed point of $Y$. 
Since $f: X\to Y$ is weakly semistable, 
we can prove that 
the diagram 
$$
\xymatrix{
X \ar@{--}[rr]^{\phi}\ar[dr]_{f}&& \widetilde X\ar[ld]^{\widetilde f}\\
&Y&
}
$$ 
behaves well under the base change by $C\hookrightarrow Y$, 
where $C$ is a general smooth curve on $Y$ passing through 
$P$. 
Roughly speaking, 
by this observation,  
we can reduce the problem to the case when $Y$ is a 
smooth projective curve. 
Note that $f$ and $\widetilde f$ are both flat. 
By Lemma \ref{f-lem3.16}, we see that 
$\dim H^0(\widetilde X_y, 
\mathcal O_{\widetilde X}(mK_{\widetilde X/Y})|_{\widetilde X_y})$ is 
independent of $y\in Y$. 
Therefore, 
$\widetilde f_*\mathcal O_{\widetilde X}(mK_{\widetilde X/Y})$ 
is locally free by the flat base change theorem. 
Thus, we obtain that 
$f_*\omega^{\otimes m}_{X/Y}$ is locally free. 
For the details, see \cite[Section 4]{fujino-direct} and 
\cite{fujino-direct-corr}.
\end{proof}

By the argument in \cite[Section 5]{fujino-direct}, 
we can prove: 

\begin{thm}[Semipositivity]\label{f-thm3.17} 
Let $f:X\to Y$ be a weakly semistable 
morphism between projective varieties with connected fibers. 
Let $m\geq 1$ be fixed. 
Assume that 
$f_*\omega^{\otimes m}_{X/Y}$ is locally free. 
Then $f_*\omega^{\otimes m}_{X/Y}$ is nef. 
\end{thm}

\begin{proof}[Idea of Proof of Theorem \ref{f-thm3.17}]

The following theorem by Popa--Schnell is 
a clever and interesting application of the Koll\'ar vanishing 
theorem:~Theorem \ref{f-thm2.5} (ii). 

\begin{thm}[{\cite[Theorem 1.4]{popa-schnell}}]\label{f-thm3.18}
Let $f:V\to W$ be a surjective morphism 
from a smooth projective 
variety $V$ onto a projective variety $W$ with 
$\dim W=n$. 
Let $\mathcal L$ be an ample line bundle 
on $W$ such that 
$|\mathcal L|$ has no base points. 
Let $k$ be a positive integer. 
Then 
$$
f_*\omega^{\otimes k}_V\otimes \mathcal L^{\otimes l}
$$ 
is generated by global sections for every 
$l\geq k(n+1)$. 
\end{thm}
By Viehweg's fiber product trick and 
the local freeness of $f_*\omega^{\otimes m}_{X/Y}$, 
we can prove that 
there exists an ample line bundle $\mathcal A$ on $Y$ such that 
$$
\left( \bigotimes ^s f_*\omega^{\otimes m}_{X/Y}\right)\otimes 
\mathcal A
$$ 
is generated by global sections for every positive integer $s$ by 
Theorem \ref{f-thm3.18}. Here, 
we used the fact that weakly semistable morphisms behave well by 
taking fiber products. 
This implies that $f_*\omega^{\otimes m}_{X/Y}$ is nef. 
For the details, see \cite[Section 5]{fujino-direct}. 
\end{proof}

As we saw above, a key ingredient of Theorem \ref{f-thm3.15} 
(resp.~Theorem \ref{f-thm3.17})  
is Koll\'ar's torsion-freeness (resp.~Koll\'ar's vanishing theorem). 
Of course, the existence of 
relative good minimal models plays a crucial role 
in the proof 
of Theorem \ref{f-thm3.15}.

\begin{rem}\label{f-rem3.19}
In the proof of Theorem \ref{f-thm3.15}, 
we need the finite generation of relative 
canonical ring $$R(X/Y)=\bigoplus _{m}^{\infty} f_*\mathcal O_X(mK_X)$$ 
from \cite{bchm} to construct a relative good minimal model of 
$f: X\to Y$. 
Note that the finite generation of $R(X/Y)$ is more or less 
Hodge theoretic when $X_{\overline \eta}$ is not of general type. 
This is because the reduction argument due to Fujino--Mori 
(see Theorem \ref{f-thm5.4} below and \cite{fujino-mori}) uses 
Theorem \ref{f-thm3.5}. 
\end{rem}

\begin{rem}\label{f-rem3.20} 
Let $V$ be a smooth 
projective variety. 
It is well-known that $V$ has a good minimal model 
when $\dim V-\kappa (V)\leq 3$. 
\end{rem}

By combining Theorem \ref{f-thm3.15} with Theorem \ref{f-thm3.17}, we obtain: 

\begin{thm}\label{f-thm3.21}
Let $f:X\to Y$ be a surjective morphism between smooth projective varieties 
with connected fibers. 
Assume that 
$$
f:X\overset{\delta}{\longrightarrow}X^\dag\overset{f^\dag}{\longrightarrow} Y
$$ 
such that $f^\dag: X^\dag\to Y$ is weakly semistable 
and that $\delta$ is a resolution of singularities. 
We further assume that 
the geometric generic fiber $X_{\overline \eta}$ of $f$ has a good 
minimal model. 
Then $f_*\omega^{\otimes m}_{X/Y}$ is a nef 
locally free sheaf for every positive integer $m$. 
\end{thm}

By the weak semistable reduction theorem 
due to Abramovich--Karu (see \cite{abramovich-karu}) 
and Theorem \ref{f-thm3.21}, we have: 

\begin{thm}\label{f-thm3.22}
Let $f:X\to Y$ be a surjective morphism between 
smooth projective varieties with connected fibers. 
Assume that the geometric generic fiber $X_{\overline \eta}$ of $f:X\to Y$ 
has a good minimal model. 
Then there exists a generically finite morphism 
$\tau:Y'\to Y$ from a smooth projective variety $Y'$ 
with the following property. 
Let $X'$ be any resolution of the main component of $X\times _Y Y'$ 
sitting in the following commutative diagram:  
$$
\xymatrix{X' \ar[r]\ar[d]_{f'}& X\ar[d]^f\\
Y'\ar[r]_{\tau} &Y. 
}
$$ 
Then $f'_*\omega^{\otimes m}_{X'/Y'}$ is a nef 
locally free sheaf for every positive integer $m$. 
\end{thm}

This means that Conjecture \ref{f-conj3.8} holds true 
under the assumption that 
the geometric generic fiber of $f$ has a good minimal model. 
More precisely, Conjecture \ref{f-conj3.8} 
follows from Conjecture \ref{f-conj3.14} 
by the weak semistable reduction theorem due to 
Abramovich--Karu (see \cite{abramovich-karu}) and Theorem \ref{f-thm3.17}. 
Moreover, Conjecture \ref{f-conj3.14} holds under the 
assumption that the geometric generic fiber has a good minimal model 
(see Theorem \ref{f-thm3.15}). 

We close this section with Takayama's result. 
Using complex analytic methods, 
Takayama in \cite{takayama} strengthened Theorem \ref{f-thm3.21} 
as follows. 

\begin{thm}[Takayama]\label{f-thm3.23} 
In Theorem \ref{f-thm3.21}, for every positive integer $m$, 
the $m$-th Narasimhan--Simha Hermitian 
metric $g_m$ on the locally free sheaf 
$\mathcal E_m =f_*\omega^{\otimes m}
_{X/Y}$ has Griffiths semipositive curvature, the induced singular Hermitian metric 
$h=e^{-\varphi}$ on $\mathcal O_{\mathbb 
P_X(\mathcal E_m)}(1)$ of $\mathbb P_X(\mathcal E_m)$ has 
semipositive curvature, and the 
Lelong number of 
the local weight $\varphi$ is zero everywhere on $\mathbb P_X(\mathcal E_m)$. 
In particular, 
$\mathcal O_{\mathbb P_X(\mathcal E_m)}(1)$ is nef. 
\end{thm}

For the definition of the Narasimhan--Simha Hermitian metric and 
the details of Theorem \ref{f-thm3.23}, see the 
original paper \cite{takayama}. We point out that 
Theorem \ref{f-thm3.23} is based on the arguments in 
\cite{fujino-direct} and \cite{fujino-direct-corr}. 

I am not so familiar with 
the analytic aspect of semipositivity theorems. 
For the details, see \cite{berndtsson}, 
\cite{bp1}, \cite{bp2}, \cite{mour}, \cite{mt1}, \cite{mt2}, 
\cite{mt3}, \cite{paun-takayama}, \cite{takayama}, etc.  

\section{Canonical divisors versus pluricanonical divisors}\label{f-sec4}

In this section, let us see that $mK_X$ with $m\geq 2$ sometimes 
behaves much better than $K_X$. We will discuss two different topics:~Koll\'ar's 
result on plurigenera in \'etale covers of 
smooth projective varieties of general type and 
Viehweg's ampleness theorem 
on direct images of relative 
pluricanonical bundles of semistable 
families of projective varieties. 
I was very much impressed by these results. 

\subsection{Plurigenera in \'etale covers}\label{f-subsec4.1}
Let us recall Koll\'ar's famous result on plurigenera in \'etale covers of smooth 
projective varieties of general type (see \cite{kollar-shafarevich}). 
For the details and some related topics, 
see also \cite[2.~Vanishing Theorems]{kollar-singularities} 
and \cite[Chapter 15]{kollar-book}. 

\begin{thm}[Koll\'ar]\label{f-thm4.1}
Let $X$ be a smooth projective variety of general type. 
Let $f:Y\to X$ be an \'etale morphism 
from a smooth projective variety $Y$. 
Then we have 
$$
h^0(Y, \mathcal O_Y(mK_Y))=\deg f \cdot h^0(X, \mathcal O_X(mK_X))
$$ 
for every positive integer $m\geq 2$. 
\end{thm}

Here, we will present Lazarsfeld's proof of Theorem \ref{f-thm4.1} following 
\cite[Theorem 11.2.23]{lazarsfeld}. 
It is actually an easy application of the theory of asymptotic multiplier ideal 
sheaves. We will give an alternative proof of Theorem \ref{f-thm4.1} 
after we discuss canonical models 
of smooth projective varieties of general type in Theorem \ref{f-thm4.5}. 

\begin{proof} 
Let $D$ be a big Cartier divisor on $X$. 
Then $\mathcal J(X, |\!| D|\!|)$ denotes the asymptotic multiplier ideal sheaf associated to 
the complete linear systems $|mD|$ for 
all $m\gg 0$. For the details of $\mathcal J(X, |\!| D |\!|)$, 
see \cite[Chapter 11]{lazarsfeld}. 
By the Nadel vanishing theorem (see \cite[Theorem 11.2.12 (ii)]{lazarsfeld}), 
$$
H^i(X, \mathcal O_X(mK_X)\otimes \mathcal J(X, |\!| (m-1)K_X|\!|))=0
$$ 
for every $i>0$ and every $m\geq 2$. 
Therefore, we have 
\begin{align*}
&h^0(X, \mathcal O_X(mK_X)\otimes \mathcal J(X, |\!| (m-1)K_X|\!|))
\\&= \chi (X, \mathcal O_X(mK_X)\otimes \mathcal J(X, |\!| (m-1)K_X|\!|))
\end{align*} 
for every $m\geq 2$. 
Since $\mathcal J(X, |\!| mK_X|\!|)\subset \mathcal J(X, |\!| (m-1)K_X|\!|)$ 
(see \cite[Theorem 11.1.8 (ii)]{lazarsfeld}), 
we have 
\begin{align*}
&H^0(X, \mathcal O_X(mK_X)\otimes \mathcal J(X, |\!| mK_X|\!|))\\ 
&= H^0(X, \mathcal O_X(mK_X)\otimes \mathcal J(X, |\!| (m-1)K_X|\!|))\\ 
&=H^0(X, \mathcal O_X(mK_X))
\end{align*} 
for every $m\geq 1$ by \cite[Proposition 11.2.10]{lazarsfeld}. 
Thus, we obtain 
$$
h^0(X, \mathcal O_X(mK_X))=\chi 
(X, \mathcal O_X(mK_X)\otimes \mathcal J(X, |\!| (m-1)K_X|\!|))
$$ 
for every $m\geq 2$. 
Similarly, 
we have 
$$
h^0(Y, \mathcal O_Y(mK_Y))=\chi 
(Y, \mathcal O_Y(mK_Y)\otimes \mathcal J(Y, |\!| (m-1)K_Y|\!|))
$$ 
for $m\geq 2$. 
Since $f$ is \'etale, $K_Y=f^*K_X$ and 
$$\mathcal J(Y, |\!| (m-1)K_Y|\!|) =f^* 
\mathcal J(X, |\!| (m-1)K_X|\!|)$$ 
by \cite[Theorem 11.2.16]{lazarsfeld}. 
Thus we have  
\begin{align*} 
&\chi (Y, \mathcal O_Y(mK_Y)\otimes \mathcal J(Y, |\!| (m-1)K_Y|\!|))
\\&= \chi (Y, f^*(\mathcal O_X(mK_X)\otimes \mathcal J (X, 
|\!|(m-1)K_X|\!|)))
\\ & =\deg f \cdot 
\chi (X, \mathcal O_X(mK_X)\otimes \mathcal J(X, |\!| (m-1)K_X|\!|))
\end{align*} 
for $m\geq 2$. 
Therefore, we obtain the desired equality 
$h^0(Y, \mathcal O_Y(mK_Y))=
\deg f \cdot h^0(X, \mathcal O_X(mK_X))$ for every $m \geq 2$. 
\end{proof}

The proof of Theorem \ref{f-thm4.1} says that 
$mK_X$ with $m\geq 2$ should be seen as $K_X+(m-1)K_X$. 
Since $m\geq 2$, $(m-1)K_X$ is big. 
Therefore, we can apply the Nadel vanishing theorem to 
\begin{align*}&\mathcal O_X(mK_X)\otimes \mathcal J(X, |\!| (m-1)K_X|\!|)\\ &=
\mathcal O_X(K_X+(m-1)K_X)\otimes \mathcal J(X, |\!| (m-1)K_X|\!|).
\end{align*} 
Obviously, the equality in Theorem \ref{f-thm4.1} does not hold 
for $m=1$. 

\begin{ex}\label{f-ex4.2} 
Let $C$ be a smooth projective curve with the genus $g(C)\geq 2$. 
Let $f: \widetilde C\to C$ be an \'etale cover with $\deg f=n\geq 2$. 
Then we have 
$$
2g(\widetilde C)-2=n(2g(C)-2)
$$ 
by Hurwitz's formula. 
This implies that 
$g(\widetilde C)=n(g(C)-1)+1$. 
Thus we have 
$$
h^0(\widetilde C, \mathcal O_{\widetilde C}(K_{\widetilde C}))
\ne n \cdot h^0(C, \mathcal O_C(K_C)). 
$$
\end{ex}

The following example also shows that 
$mK_X$ with $m\geq 2$ sometimes contains much more information 
than $K_X$. 

\begin{ex}[Godeaux surface]\label{f-ex4.3} 
We put $$
Y=(Z^5_0+Z^5_1+Z^5_2+Z^5_3=0)\subset \mathbb P^3. 
$$
Then $Y$ is a smooth projective 
surface such that 
$$
\mathcal O_Y(K_Y)=\mathcal O_{\mathbb P^3}(-4+5)|_Y=\mathcal O_Y(1)
$$ 
is very ample. 
Therefore, $Y$ is of general type, 
$$
h^0(Y, \mathcal O_Y(K_Y))=h^0(\mathbb P^3, \mathcal O_{\mathbb P^3}(1))=4, 
$$ 
and $q(Y)=h^1(Y, \mathcal O_Y)=0$. 
We put $G=\mathbb Z/ 5 \mathbb Z$. 
Then $G$ acts freely on $Y$ by 
$$
[Z_0: Z_1: Z_2: Z_3] \mapsto 
[Z_0: \zeta Z_1: \zeta^2 Z_2: \zeta^3 Z_3]
$$ 
where $\zeta=\exp \left (\frac{2\pi \sqrt{-1}}{5}\right)$. 
We put $X=Y/G$. 
Then $X$ is a smooth projective surface with 
ample canonical divisor. 
Let $f:Y\to X$ be the natural map. 
Then $f$ is a finite \'etale morphism. 
We can directly check that $h^0(X, \mathcal O_X(K_X))=0$ by 
$H^0(X, \mathcal O_X(K_X))=H^0(Y, \mathcal O_Y(K_Y))^G$. 
Note that $q(X)=h^1(X, \mathcal O_X)=0$, 
$K_Y^2=5$, and $K_X^2=1$. 
We also note that 
$X$ is known as a Godeaux surface. 
It is well-known that 
the linear system $|mK_X|$ gives an embedding 
into a projective space for every $m\geq 5$. 
Note that 
$$
4=h^0(Y, \mathcal O_Y(K_Y))\ne \deg f \cdot 
h^0(X, \mathcal O_X(K_X))=0. 
$$
\end{ex}

\begin{say}[Canonical models]\label{f-say4.4} 
Let us discuss canonical models 
of finite \'etale covers of smooth projective varieties of general type. 
The existence of 
canonical models was unkonwn when \cite{kollar-shafarevich} was written. 
Let $\pi:V\to W$ be a projective 
surjective morphism 
from a smooth quasiprojective 
variety $V$ onto a quasiprojective 
variety $W$. 
Assume that 
$K_V$ is $\pi$-big. 
Then, by \cite{bchm}, 
the (relative) canonical 
ring 
$$
R(V/W)=\bigoplus _{m=0}^\infty \pi_*\mathcal O_V(mK_V)
$$ 
is a finitely generated 
$\mathcal O_W$-algebra. 
We put 
$$
V_c=\mathrm{Proj} _W R(V/W)
$$ 
and call it the canonical model of $V$ over $W$ or 
the relative canonical model of $\pi:V\to W$. 
It is well-known that 
$V_c$ is birationally equivalent to 
$V$ over $W$, 
$V_c$ has only canonical singularities, 
and 
$K_{V_c}$ is ample over $W$. 
\end{say}

A finite \'etale morphism between smooth 
projective varieties of general type induces a natural finite 
\'etale morphism 
between their canonical models:  

\begin{thm}\label{f-thm4.5}
Let $f:Y\to X$ be a finite \'etale morphism 
between smooth projective varieties of general type. 
Let $X_c$ be the canonical model of $X$ and 
let $Y_c$ be the canonical model 
of $Y$. 
Then there exists a finite \'etale morphism 
$f_c: Y_c\to X_c$ such that 
$$
\xymatrix{
Y \ar[d]_{f}\ar@{-->}[r]& Y_c\ar[d]^{f_c} \\
X \ar@{-->}[r]& X_c 
}
$$ is commutative. 
\end{thm}

The following proof was suggested by Yoshinori Gongyo. 

\begin{proof}
By taking an elimination of indeterminacy of 
$X\dashrightarrow X_c$ and 
the base change of $Y$, we may assume that 
$g:X\to X_c$ is a morphism. 
Let $Y'\to X_c$ be the relative canonical model of 
$g\circ f: Y\to X_c$ (see \cite{bchm} and 
\ref{f-say4.4}). 
Then, by using the negativity lemma (see, for example, 
\cite[Lemma 3.39]{kollar-mori}), 
we see that $K_{Y'}=f'^*K_{X_c}$ and 
that $f': Y'\to X_c$ is finite 
since $K_{Y'}$ is $f'$-ample. 
Therefore, $K_{Y'}$ is ample since $K_{X_c}$ is ample. 
This implies that 
$Y'=Y_c$, that is, 
$Y'$ is the canonical model 
of $Y$. 
Note that $Y$ is smooth and is finite over $X$. 
Therefore, $Y$ is the normalization of the 
main component of $X\times _{X_c}Y_c$. 
Thus we obtain the following commutative diagram: 
$$
\xymatrix{
Y \ar[d]_f \ar[r]& Y_c\ar[d]^{f_c}\\ 
X \ar[r]_g& X_c. 
}
$$
By Lemma \ref{f-lem4.8} below, 
we obtain that $f_c$ is a finite \'etale morphism. 
\end{proof}

By the proof of Theorem \ref{f-thm4.5}, 
we have: 

\begin{thm}\label{f-thm4.6}
Let $f:Y\to X$ be a finite \'etale morphism 
between smooth 
projective varieties. 
Let $X_m$ be a minimal model 
of $X$. 
Then we can construct a commutative diagram: 
$$
\xymatrix{
Y \ar[d]_{f}\ar@{-->}[r]^{\varphi}& \widetilde Y\ar[d]^{\widetilde f} \\
X \ar@{-->}[r]& X_m
}
$$ 
such that 
$\widetilde f: \widetilde Y\to X_m$ is a finite 
\'etale morphism and that $\varphi$ is birational. 
\end{thm}

For the definition of minimal models, 
recall Definition \ref{f-def3.12}. 

\begin{proof}
As in the proof of Theorem \ref{f-thm4.5}, 
we may assume that $g:X\to X_m$ is a morphism. 
Let $\widetilde f: \widetilde Y\to X_m$ be the relative 
canonical model of $g\circ f: Y\to X_m$. 
Then, by the proof of Theorem \ref{f-thm4.5}, 
$\varphi:Y\dashrightarrow \widetilde Y$ and 
$\widetilde f: \widetilde Y\to X_m$ satisfy the 
desired properties. 
\end{proof}

Related to Theorem \ref{f-thm4.5} and Theorem \ref{f-thm4.6}, 
we have: 

\begin{problem}\label{f-prob4.7}
Find a projective variety $X$ such that 
$X$ has only $\mathbb Q$-factorial 
terminal (or canonical) singularities with nef (or ample) canonical divisor and 
a finite \'etale morphism $f:Y\to X$ such that 
$Y$ is not $\mathbb Q$-factorial. 
\end{problem}

The following lemma is a special case of \cite[Lemma 3.9]{nakayama-zhang}. 

\begin{lem}\label{f-lem4.8}
We consider the following commutative diagram of 
quasiprojective varieties: 
$$
\xymatrix{
Y \ar[r]^	q\ar[d]_f& W \ar[d]^h\\ 
X \ar[r]_p& V 
}
$$ 
such that 
\begin{itemize}
\item[(i)] $X$ and $Y$ are smooth, 
\item[(ii)] $p$ and $q$ are projective birational morphisms, 
\item[(iii)] $V$ and $W$ are normal and have only rational singularities, 
\item[(iv)] $f$ is a finite \'etale morphism, and 
\item[(v)] $h$ is finite. 
\end{itemize} 
Then $h$ is an \'etale morphism. 
\end{lem}

We give a proof for the reader's convenience. 
It is interesting for me that 
the proof of Lemma \ref{f-lem4.8} below uses the 
$E_1$-degeneration of 
Hodge to de Rham type spectral sequences 
for projective 
simple normal crossing varieties. 
 
\begin{proof}
We take an arbitrary point $P\in V$. 
By taking a birational modification of $X$ and 
the base change of $Y$, we may assume that $E=\Supp (p^{-1}(P))$ 
is a simple normal crossing divisor on $X$. 
Since $f$ is \'etale, $f^{-1}(E)$ is a simple normal crossing divisor on $Y$. 
Note that $f^{-1}(E)$ is the disjoint union of $E_Q=\Supp (q^{-1}(Q))$ for 
points $Q\in h^{-1}(P)$. 
Since $V$ has only rational singularities, 
$R^ip_*\mathcal O_X=0$ for every $i>0$. 

\begin{claim}
The natural map 
$$
\pi:H^i(E, \mathbb C)\to H^i(E, \mathcal O_E)
$$ 
induced by the inclusion 
$\mathbb C_E\hookrightarrow \mathcal O_E$ is surjective for every $i$. 
\end{claim}

\begin{proof}[Proof of Claim]
By using the Mayer--Vietoris simplicial resolution of a projective 
simple normal crossing variety $E$, 
we can construct a cohomological mixed Hodge complex 
$(K_{\mathbb Z}, (K_{\mathbb Q}, W_{\mathbb Q}), (K_{\mathbb C}, W, F))$ 
which induces a natural mixed Hodge  structure on $H^{\bullet}(E, \mathbb Z)$. 
By the theory of mixed Hodge structures, we see that  
$$
E^{p, q}_1=\mathbb H^{p+q}(E, \mathrm{Gr}^p_FK_{\mathbb C})\Rightarrow 
H^{p+q}(E, \mathbb C)
$$ 
degenerates at $E_1$. 
We can directly check that $\mathrm{Gr}^0_F K_{\mathbb C}$ is quasi-isomorphic 
to $\mathcal O_E$ by using the Mayer--Vietoris simplicial resolution of $E$. 
Thus, 
we obtain that 
$$
\pi:H^i(E, \mathbb C)\to H^i(E, \mathcal O_E)
$$ 
is surjective for every $i$. 
\end{proof}

By the following commutative diagram: 
$$
\xymatrix{
(R^ip_*\mathbb C_X)_P \ar[d]\ar[r]^{\simeq} & H^i(E, \mathbb C)\ar[d]^{\pi}\\
(R^ip_*\mathcal O_X)_P\ar[r]&H^i(E, \mathcal O_E), 
}
$$ 
we obtain that $H^i(E, \mathcal O_E)=0$ for every $i>0$, 
as $R^ip_*\mathcal O_X=0$ for 
every $i>0$. 
Thus, we obtain $\chi (E, \mathcal O_E)=1$. 
By the same argument, 
we have $\chi (E_Q, \mathcal O_{E_Q})=1$. 
On the other hand, 
$$
\chi (f^{-1}(E), \mathcal O_{f^{-1}(E)})=\deg f\cdot \chi (E, \mathcal O_E)=\deg f
$$ 
since $f$ is \'etale. 
Therefore, we have 
$$
\sharp h^{-1}(P)=\sum _{Q\in h^{-1}(P)}\chi( E_Q, \mathcal O_{E_Q})
=\chi (f^{-1}(E), \mathcal O_{f^{-1}(E)})=\deg f. 
$$ 
This implies that 
$f: E_Q\to E$ is an isomorphism for every $Q\in h^{-1}(P)$. 
Thus, by the theorem on formal functions, we have 
$$
\widehat {\mathcal O}_{V, P}=(p_*\mathcal O_X)^{\wedge}_P
\simeq (q_*\mathcal O_Y)^{\wedge}_Q=\widehat{\mathcal O}_{W, Q}
$$ 
for every $Q\in h^{-1}(P)$. 
Thus, $h$ is an \'etale morphism. 
\end{proof}

We give an alternative proof of 
Theorem \ref{f-thm4.1}, 
one that is natural from the minimal model theoretic viewpoint. 
Unfortunately, it may be 
more complicated than Lazarsfeld's proof given before. 

\begin{proof}[Proof of Theorem \ref{f-thm4.1}]
By Theorem \ref{f-thm4.5}, 
we may replace $f:Y\to X$ with $f_c: Y_c\to X_c$. 
Note that 
$$
h^0(X_c, \mathcal O_{X_c}(mK_{X_c}))=h^0(X, \mathcal O_X(mK_X))
$$
and 
$$
h^0(Y_c, \mathcal O_{Y_c}(mK_{Y_c}))=h^0(Y, \mathcal O_Y(mK_Y))
$$ 
for every $m\geq 1$. 
By the Kawamata--Viehweg vanishing theorem for 
singular varieties (see, for example, \cite[Corollary 5.7.7]{fujino-foundation}), 
we have 
$$
H^i(X_c, \mathcal O_{X_c}(mK_{X_c}))=H^i(Y_c, \mathcal O_{Y_c}(mK_{Y_c}))
=0
$$ 
for every $i>0$ and every $m\geq 2$. 
We also note that 
$f^*_c\mathcal O_{X_c}(mK_{X_c})=\mathcal O_{Y_c}(mK_{Y_c})$ 
for every $m$ since $f_c$ is \'etale. 
Therefore, we obtain 
\begin{align*}
h^0(Y_c, \mathcal O_{Y_c}(mK_{Y_c}))&=\chi 
(Y_c, \mathcal O_{Y_c}(mK_{Y_c})) 
\\& =\deg f_c\cdot \chi (X_c, \mathcal O_{X_c}(mK_{X_c})) 
\\& =\deg f_c \cdot h^0(X_c, \mathcal O_{X_c}(mK_{X_c}))
\end{align*} 
for every $m\geq 2$. 
This implies the desired equality. 
\end{proof}

\begin{rem}\label{f-rem4.9}
In Lemma \ref{f-lem4.8}, the assumption that $V$ and $W$ have only 
rational singularities (see (iii) in Lemma \ref{f-lem4.8}) is indispensable. 

Consider the following example. 
Let $C\subset \mathbb P^2$ be an elliptic curve and 
let $V\subset \mathbb P^3$ be a cone over $C\subset \mathbb P^2$. 
Let $p: X\to V$ be the blow-up at the vertex $P$ of $V$ and 
let $E$ be the $p$-exceptional 
divisor on $X$. 
Note that there is a natural $\mathbb P^1$-bundle structure 
$\pi:X\to C$ and $E$ is a section of $\pi$. 
We take a nontrivial finite \'etale cover $D\to C$. 
We put $Y=X\times _C D$ and $F=E\times _C D$. 
Let $H$ be an ample Cartier divisor on $V$. 
We consider $q=\Phi_{|mf^*p^*H|}: Y\to W$ for a sufficiently 
large positive integer $m$. 
Note that $q$ contracts $F$ to an isolated 
normal singular point $Q$ of $W$. 
Then we have the following 
commutative diagram 
$$
\xymatrix{
Y \ar[r]^{q}\ar[d]_f & W \ar[d]^{h}\\ 
X\ar[r]_p& V
}
$$
such that $f$ is \'etale, $h$ is finite, but $h$ is not \'etale. 
Note that $h^{-1}(P)=Q$ since $f^{-1}(E)=F$. 
It is also false that 
the singularities of $V$ and $W$ are rational.   
\end{rem}

\subsection{Viehweg's ampleness theorem}\label{f-subsec4.2}
We treat direct images of relative pluricanonical bundles. 
The following theorem is buried in Viehweg's papers (see \cite{viehweg1} and 
\cite{viehweg2}). 
The statement seems to be magical. 

\begin{thm}[Viehweg]\label{f-thm4.10} 
Let $f:X\to Y$ be a surjective morphism from a smooth projective variety 
$X$ onto a smooth projective curve $Y$ with connected fibers. 
Then $f_*\omega^{\otimes m}_{X/Y}$ is nef for every 
positive integer $m$. In particular, we have  
$\deg \det f_*\omega^{\otimes m}_{X/Y}\geq 0$ for 
every positive integer $m$. 
Assume that $f$ is semistable. 
If $\deg \det f_*\omega^{\otimes k}_{X/Y}>0$, that is, 
$\det f_*\omega^{\otimes k}_{X/Y}$ is ample, for 
some positive integer $k$, then 
$f_*\omega^{\otimes k'}_{X/Y}$ is ample, where 
$k'$ is any multiple of $k$ with $k'\geq 2$. 
\end{thm}

\begin{proof}
From Kawamata (see \cite[Theorem 1]{kawamata-curves}), 
we have that $f_*\omega^{\otimes m}_{X/Y}$ is nef for every $m\geq 1$. 
Or, by Viehweg's weak positivity:~\cite[Theorem III]{viehweg1} (see 
also \cite[Theorem 4.3 and Theorem 5.5]{fujino-subadditivity} and 
Remark \ref{f-rem4.14} below), 
$f_*\omega^{\otimes m}_{X/Y}$ is weakly positive for every $m\geq 1$. 
Since $Y$ is a smooth projective curve, 
the weak positivity implies that 
$f_*\omega^{\otimes m}_{X/Y}$ is nef for every $m\geq 1$. 
By \cite[Theorem 3.5]{viehweg2} (see also 
\cite[Theorem 5.11]{fujino-subadditivity}), 
$\deg \det f_*\omega^{\otimes k}_{X/Y}>0$ implies that 
$f_*\omega^{\otimes k'}_{X/Y}$ is big in the sense of Viehweg 
(see \cite[Definition 3.1]{fujino-subadditivity}) when $f$ is semistable. 
Then, by \cite[Lemma 3.6]{viehweg2} (see also \cite[Lemma 3.7]{fujino-subadditivity}), 
there is a generically isomorphic injection 
$$
\bigoplus _r \mathcal A \to S^\nu (f_*\omega^{\otimes k'}_{X/Y})
$$ 
for some ample invertible sheaf $\mathcal A$ on $Y$ and 
some positive integer $\nu$, where 
$r=\mathrm{rank}\, S^\nu (f_*\omega^{\otimes k'}_{X/Y})$. 
This implies that $S^\nu (f_*\omega^{\otimes k'}_{X/Y})$ is ample by \cite[Theorem 
6.4.15]{lazarsfeld}. 
Therefore, $f_*\omega^{\otimes k'}_{X/Y}$ is ample by 
\cite[Proposition (2.4)]{hartshorne-ample}. 
\end{proof}

Roughly speaking, Theorem \ref{f-thm4.10} says if 
$f_*\omega_{X/Y}$ is a little bit positive then 
$f_*\omega^{\otimes m}_{X/Y}$ is very positive 
for $m\geq 2$. Viehweg's arguments in \cite{viehweg1} and 
\cite{viehweg2} (see also \cite{fujino-subadditivity}) use 
his clever covering trick and 
fiber product trick. 
They are geometric. 
It seems to be very important to find a more direct approach 
to Theorem \ref{f-thm4.10}. Thus, we have: 

\begin{problem}\label{f-prob4.11} 
Find an analytic (and more direct) proof of Theorem \ref{f-thm4.10}. 
\end{problem}

The example by Catanese--Dettweller (see \cite{catanese1}, 
\cite{catanese2}, and \cite{catanese3}) 
below says that the condition $k'\geq 2$ in Theorem \ref{f-thm4.10} 
is indispensable. 

\begin{ex}[Catanese--Dettweller]\label{f-ex4.12} 
There exist a smooth projective 
surface $X$ of general type and a smooth 
projective curve $Y$ such that 
$f:X\to Y$ is semistable and 
that $f_*\omega_{X/Y}=A\oplus Q$, where 
$A$ is an ample vector bundle of rank $2$ and $Q$ 
is a unitary flat vector bundle of rank $4$. 
Moreover, $Q$ is not semiample.  
In this case, $\deg \det f_*\omega_{X/Y}>0$. 
By Theorem \ref{f-thm4.10}, 
$f_*\omega^{\otimes m}_{X/Y}$ is ample for every $m\geq 2$. 
However, 
$f_*\omega_{X/Y}$ is not ample. 
\end{ex}

Note that the construction of Example \ref{f-ex4.12} in \cite{catanese2} 
depends on the theory of variation of Hodge structure. 
For the details, see \cite{catanese1}, 
\cite{catanese2}, and \cite{catanese3}. 

\begin{problem}\label{f-prob4.13}
Find similar examples to Example \ref{f-ex4.12} 
that do not use the theory of variation of Hodge structure. 
\end{problem}

Although we do not discuss Viehweg's weak positivity in this paper, 
we give a remark on the weak positivity of 
$f_*\omega^{\otimes m}_{X/Y}$ for the interested reader:  

\begin{rem}\label{f-rem4.14}
Let $f:X\to Y$ be a surjective 
morphism between smooth projective 
varieties with connected fibers. 
Then $f_*\omega^{\otimes m}_{X/Y}$ is weakly positive 
for every positive integer $m$ by 
Viehweg (see \cite{viehweg1}). 
Viehweg's original proof of his weak positivity 
uses Theorem \ref{f-thm3.5}. 
Given what we know now, 
we can prove the weak positivity of $f_*\omega^{\otimes m}_{X/Y}$ by Koll\'ar's 
vanishing theorem:~Theorem \ref{f-thm2.5} (ii). 
Moreover, Theorem \ref{f-thm3.18} drastically simplifies the proof of 
Viehweg's 
weak positivity of $f_*\omega^{\otimes m}_{X/Y}$. 
For the details, see \cite{fujino-subadditivity}. 
\end{rem}

Anyway, we should consider not only $K_X$ but also $mK_X$ with 
$m\geq 2$ in order to 
understand complex projective varieties much better. 

\section{On finite generation of (log) canonical rings}\label{f-sec5}

In this section, we quickly discuss the finite 
generation of (log) canonical rings due to 
Birkar--Cascini--Hacon--M\textsuperscript{c}Kernan 
(see \cite{bchm}) and some related topics (see \cite{fujino-mori} and 
\cite{fujino-some}). 
For simplicity, we only treat the absolute setting in this section. 
However, the relative setting is very important and 
is indispensable for 
some applications (cf.~Remark \ref{f-rem3.19}). 

\begin{thm}\label{f-thm5.1} 
Let $X$ be a smooth projective variety and let $\Delta$ be an effective 
$\mathbb Q$-divisor on $X$ such that 
$\Supp \Delta$ is a simple normal crossing divisor and 
that the coefficients of $\Delta$ are less than one. 
Assume that $K_X+\Delta$ is big. 
Then the log canonical ring 
$$
R(X, K_X+\Delta)=\bigoplus _{m=0} ^\infty H^0(X, \mathcal O_X(
\lfloor m(K_X+\Delta)\rfloor))
$$ 
is a finitely generated $\mathbb C$-algebra. 
\end{thm}

Theorem \ref{f-thm5.1} was first obtained in \cite{bchm}. 
For the proof of Theorem \ref{f-thm5.1}, 
see also \cite{cascini-lazic} and \cite{paun2}. 
We know that 
the proof of Theorem \ref{f-thm5.1} was greatly influenced by 
Siu's extension argument (see \cite{siu1}) based on the 
Ohsawa--Takegoshi 
$L^2$ extension theorem 
(see \cite{hacon-mckernan}, \cite{cascini-lazic}, and \cite{paun2}) 
although \cite{hacon-mckernan} and \cite{cascini-lazic} only 
use the Kawamata--Viehweg 
vanishing theorem and do not use $L^2$ methods. 
I feel that Theorem \ref{f-thm5.1} is not Hodge theoretic. 
It is natural to see that Theorem \ref{f-thm5.1} is more closely 
related to $L^2$ methods than to Hodge theory. 

By combining Theorem \ref{f-thm5.1} with 
the result in \cite{fujino-mori}, we have: 

\begin{thm}\label{f-thm5.2} 
Let $X$ be a smooth projective variety and let $\Delta$ be an effective 
$\mathbb Q$-divisor on $X$ such that 
$\Supp \Delta$ is a simple normal crossing divisor and 
that the coefficients of $\Delta$ are less than one. 
Then the log canonical ring 
$$
R(X, K_X+\Delta)=\bigoplus _{m=0} ^\infty H^0(X, \mathcal O_X(
\lfloor m(K_X+\Delta)\rfloor))
$$ 
is a finitely generated $\mathbb C$-algebra. 
\end{thm}

Note that we do not 
assume that $K_X+\Delta$ is big in Theorem \ref{f-thm5.2}. 
As a corollary of Theorem \ref{f-thm5.2}, we get when $\Delta=0$: 

\begin{cor}\label{f-cor5.3}
Let $X$ be a smooth projective variety. Then 
the canonical ring 
$$
R(X)=\bigoplus _{m=0} ^\infty H^0(X, \mathcal O_X(mK_X))
$$ 
is a finitely generated $\mathbb C$-algebra. 
\end{cor}

We note that the formulation of Theorem \ref{f-thm5.1}, which 
may look artificial, is indispensable for 
the proof of Corollary \ref{f-cor5.3}. 
We will see how to reduce Theorem \ref{f-thm5.2} to Theorem \ref{f-thm5.1}. 
I think that this reduction step is 
more or less Hodge theoretic. 
I do not know how to prove Corollary \ref{f-cor5.3} 
without using this reduction argument based on the Fujino--Mori 
canonical bundle formula (see \cite{fujino-mori}).  
We can prove the following theorem from \cite{fujino-mori}. 

\begin{thm}[Fujino--Mori]\label{f-thm5.4} 
Let $X$ be a smooth projective variety, let $\Delta$ be 
an effective 
$\mathbb Q$-divisor on $X$ such that 
$\Supp \Delta$ is a simple normal crossing divisor with 
$\lfloor \Delta\rfloor =0$. 
Let $f:X\to Y$ be a surjective morphism between smooth 
projective varieties with connected fibers. 
Assume that 
$\kappa (X_{\overline{\eta}}, 
K_{X_{\overline {\eta}}}+\Delta|_{X_{\overline{\eta}}})=0$ 
where $X_{\overline {\eta}}$ is the 
geometric generic fiber of $f:X\to Y$. 
Then we can construct a commutative diagram: 
$$
\xymatrix{X\ar[d]_{f}& X'\ar[l]_p\ar[d]^{f'}\\
Y & Y'\ar[l]^q
}
$$
with the following properties. 
\begin{itemize}
\item[(i)] $p$ and $q$ are projective birational morphisms. 
\item[(ii)] $X'$ and $Y'$ are smooth projective 
varieties. 
\item[(iii)] There exists an effective $\mathbb Q$-divisor 
$\Delta'$ on $X'$ such that 
$\Supp \Delta'$ is a simple normal crossing divisor on $X'$ 
with $\lfloor \Delta'\rfloor=0$ and that 
$$
H^0(X, \mathcal O_X(\lfloor m(K_X+\Delta)\rfloor))\simeq 
H^0(X', \mathcal O_{X'}(\lfloor m(K_{X'}+\Delta')\rfloor))
$$ 
for every positive integer $m$. 
\item[(iv)] There exist a positive integer $k$, a 
nef $\mathbb Q$-divisor 
$M$ on $Y'$, and an effective $\mathbb Q$-divisor $D$ on $Y'$ such that 
$\Supp D$ is a simple normal crossing divisor with $\lfloor D\rfloor =0$ and 
that 
\begin{align*}
&H^0(X', \mathcal O_{X'}(\lfloor mk(K_{X'}+\Delta')\rfloor))
\\&\simeq H^0(Y', \mathcal O_{Y'}(\lfloor mk(K_{Y'}+M+D)\rfloor))
\end{align*}
for every positive integer $m$. 
\end{itemize}
\end{thm}

We do not give a proof of Theorem \ref{f-thm5.4} here. 
For the details, see \cite[Theorem 4.5]{fujino-mori}. 

\begin{rem}\label{f-rem5.5}
Theorem \ref{f-thm5.4} is an application of 
the Fujino--Mori canonical bundle formula discussed in \cite{fujino-mori}. 
The $\mathbb Q$-divisor $M$ is called the semistable part in 
\cite{fujino-mori} and now is usually called the moduli part in 
the literature. 
Note that the nefness of $M$ comes from Theorem \ref{f-thm3.5}. 
Therefore, Theorem \ref{f-thm5.4} is more or less 
Hodge theoretic. 
\end{rem}

Using Kodaira's lemma and Hironaka's resolution theorem, 
we can prove: 

\begin{prop}\label{f-prop5.6}
If $K_{Y'}+M+D$ is big in Theorem \ref{f-thm5.4}, then 
there exist a birational morphism 
$r:Z\to Y'$ from a smooth 
projective variety $Z$, an effective $\mathbb Q$-divisor 
$\Delta_Z$ on $Z$, and positive integers $a$ and $b$ such that 
$\Supp \Delta_Z$ is a simple normal crossing divisor with 
$\lfloor \Delta_Z\rfloor=0$ and 
that 
\begin{align*}
&H^0(Y', \mathcal O_{Y'}(\lfloor ma(K_{Y'}+M+D)\rfloor))
\\ & \simeq H^0(Z, \mathcal O_Z(\lfloor mb(K_Z+\Delta_Z)\rfloor))
\end{align*} 
for every positive integer $m$. 
\end{prop}

\begin{proof}
This is an easy consequence of Kodaira's lemma on big $\mathbb Q$-divisors 
and Hironaka's resolution of singularities. 
We note that we can choose $\Delta_Z$ with $\lfloor \Delta_Z\rfloor =0$ since 
$M$ is nef and $\lfloor D\rfloor =0$. 
More precisely, by Kodaira's lemma, we have 
$K_{Y'}+M+D\sim _{\mathbb Q} A+E$, where $A$ is an ample 
$\mathbb Q$-divisor and $E$ is an effective $\mathbb Q$-divisor. 
Thus we have 
$(1+\varepsilon)(K_{Y'}+M+D)\sim _{\mathbb Q}K_{Y'}+(M+\varepsilon A)+D+\varepsilon 
E$ for every positive rational number $\varepsilon$. 
If $\varepsilon$ is sufficiently small, 
then $(Y', D+\varepsilon E)$ is kawamata log terminal. 
Since $M+\varepsilon A$ is ample, 
we can take an effective 
$\mathbb Q$-divisor $B$ such that 
$B\sim _{\mathbb Q} M+\varepsilon A$ and that $(Y', \Delta_{Y'})$ is still 
kawamata log terminal, where $\Delta_{Y'}=B+D+\varepsilon E$. 
Therefore, we can find positive integers $a$ and $b$ such that 
$a(K_{Y'}+M+D)\sim b(K_{Y'}+\Delta_{Y'})$. 
By Hironaka's resolution theorem, 
we can take $r:Z\to Y'$ and $\Delta_{Z}$ such that 
$\Supp \Delta_Z$ is a simple normal crossing divisor 
with $\lfloor \Delta_Z\rfloor =0$ and that 
$H^0(Z, \mathcal O_Z(\lfloor m(K_Z+\Delta_Z)\rfloor))\simeq 
H^0(Y', \mathcal O_{Y'}(\lfloor m(K_{Y'}+\Delta_{Y'})\rfloor))$ for 
every positive integer $m$. Thus we have the desired properties. 
\end{proof}

Let us see how to prove Theorem \ref{f-thm5.2} by using 
Theorem \ref{f-thm5.1}, Theorem \ref{f-thm5.4}, and 
Proposition \ref{f-prop5.6}. 

\begin{proof}[Proof of 
Theorem \ref{f-thm5.2}] Let $X$ and $\Delta$ be as in Theorem \ref{f-thm5.2}. 
Assume that $\kappa (X, K_X+\Delta)\geq 0$ and that 
$K_X+\Delta$ is not big. 
We consider the Iitaka fibration $\Phi _{|m(K_X+\Delta)|}: 
X\dashrightarrow  Y$ for some large and divisible positive integer $m$. 
By taking a suitable birational modification, 
we may assume that $f:X\to Y$ is a surjective morphism 
between smooth projective varieties with connected 
fibers. 
Then we apply Theorem \ref{f-thm5.4} and 
Proposition \ref{f-prop5.6} to $f:X\to Y$. 
Thus, we see that $R(X, K_X+\Delta)$ is a finitely generated 
$\mathbb C$-algebra if and only if 
$R(Z, K_Z+\Delta_Z)$ is a finitely generated $\mathbb C$-algebra. 
By Theorem \ref{f-thm5.1}, we know that 
$R(Z, K_Z+\Delta_Z)$ is a finitely generated $\mathbb C$-algebra. 
Therefore, we obtain Theorem \ref{f-thm5.2}.  
\end{proof}

As I explained in Remark \ref{f-rem5.5}, 
Theorem \ref{f-thm5.4} is more or less Hodge theoretic. 
So, we pose: 

\begin{problem}\label{f-prob5.7}
Prove Corollary \ref{f-cor5.3} without using Hodge theory. 
\end{problem}

\begin{rem}\label{f-rem5.8}
Theorem \ref{f-thm3.5} holds true under the assumption that 
$X$ is only a compact K\"ahler manifold. 
This is because we can use the theory of variation of Hodge 
structure even for compact K\"ahler manifolds.  
Therefore, Theorem \ref{f-thm5.4} holds true 
under the assumption that 
$X$ is a compact complex manifold in Fujiki's class $\mathcal C$. 
As a consequence, we see that Theorem \ref{f-thm5.2} holds 
for compact complex manifolds in Fujiki's class $\mathcal C$. 
For the details, see \cite[Section 5]{fujino-some}. 
As a special case, we have the corollary below. 
\end{rem}

\begin{cor}\label{f-cor5.9}
Let $X$ be a compact K\"ahler manifold. 
Then the canonical ring 
$$
R(X)=\bigoplus _{m=0}^\infty H^0(X, \omega^{\otimes m}_X)
$$ 
is a finitely generated $\mathbb C$-algebra. 
\end{cor}

\begin{rem}\label{f-rem5.10}
There exists a compact complex non-K\"ahler manifold whose 
canonical ring is not a finitely generated $\mathbb C$-algebra (see 
\cite[Example 6.4]{fujino-some}, which is 
essentially due to Wilson \cite{wilson}). 
This means that Corollary \ref{f-cor5.9} does not hold for 
compact complex non-K\"ahler manifolds. 
The reader can find a smooth morphism 
$f:X\to Y$ from a compact complex non-K\"ahler manifold $X$ onto 
$Y=\mathbb P^1$ with 
connected fibers such that 
$f_*\omega_{X/Y}\simeq \mathcal O_{\mathbb P^1}(-2)$ 
(see \cite[Example 6.1]{fujino-some}). 
This example is due to Atiyah. 
Therefore, Theorem \ref{f-thm3.5} does not hold for 
compact non-K\"ahler manifolds. 
Of course, by this example, we see that 
Theorem \ref{f-thm5.4} does not hold true without assuming that 
$X$ is a compact complex manifold in Fujiki's class $\mathcal C$. 
For the details, see \cite{fujino-some}. 
\end{rem}

As a generalization of Theorem \ref{f-thm5.2}, we have 
Conjecture \ref{f-conj5.11}, where we are allowing the coefficients of 
$\Delta$ to equal one.  

\begin{conj}\label{f-conj5.11} 
Let $X$ be a smooth projective variety 
and let $\Delta$ be an effective $\mathbb Q$-divisor 
on $X$ such that $\Supp \Delta$ is a simple 
normal crossing divisor and that 
the coefficients of $\Delta$ are less than or equal to 
one. 
Then the log canonical 
ring 
$$
R(X, K_X+\Delta)=\bigoplus _{m=0}^\infty 
H^0(X, \mathcal O_X(\lfloor 
m(K_X+\Delta)\rfloor))
$$ 
is a finitely generated $\mathbb C$-algebra. 
\end{conj}

\begin{rem}\label{f-rem5.12} 
Of course, we think that Conjecture \ref{f-conj5.11} also holds for compact 
complex manifolds in Fujiki's class $\mathcal C$. 
However, I could not reduce Conjecture \ref{f-conj5.11} for 
compact complex manifolds in Fujiki's class $\mathcal C$ to 
Conjecture \ref{f-conj5.11} for projective varieties. 
Therefore, Conjecture \ref{f-conj5.11} for compact complex manifolds 
in Fujiki's class $\mathcal C$ may be much harder than for 
projective varieties. For the details, see \cite{fujino-some}. 
\end{rem}

I have already discussed Conjecture \ref{f-conj5.11} in detail 
in a joint paper with Yoshinori Gongyo (see \cite{fujino-gongyo}). 
In \cite{fujino-gongyo}, we clarified the relationship among 
various conjectures in the minimal model program related 
to Conjecture \ref{f-conj5.11}. So we do not repeat 
the discussions about Conjecture \ref{f-conj5.11} here. 
We strongly recommend the interested reader to see \cite{fujino-gongyo}. 

In order to prove Conjecture \ref{f-conj5.11}, the following 
famous conjecture seems to be unavoidable. 
I think that it is a very difficult open problem for higher-dimensional 
algebraic varieties. 

\begin{conj}[Nonvanishing conjecture]\label{f-conj5.13}
Let $X$ be a smooth projective variety such that 
the canonical divisor $K_X$ is pseudoeffective. 
Then we have $H^0(X, \mathcal O_X(mK_X))\ne 0$ for 
some positive integer $m$, equivalently, 
$\kappa (X)\geq 0$. 
\end{conj}

For the reader's convenience, let us recall the definition of 
pseudoeffective divisors. 

\begin{defn}\label{f-def5.14}
Let $X$ be a smooth projective variety and let $D$ be a 
Cartier divisor on $X$. 
Then $D$ is pseudoeffective if $D+A$ is big for every ample 
$\mathbb Q$-divisor $A$ on $X$.  
\end{defn}

The characterization of pseudoeffective divisors via singular 
hermitian metrics may be helpful. 

\begin{rem}\label{f-rem5.15}
Let $X$ be a smooth projective variety and let $D$ be a Cartier divisor on $X$. 
Then $D$ is pseudoeffective if and only if 
$\mathcal O_X(D)$ has a singular hermitian metric 
$h$ with 
$\sqrt{-1}\Theta_h\geq 0$ in the sense of currents.  
\end{rem}

The characterization of uniruled varieties 
due to Boucksom--Demailly--P\u aun--Peternell (see \cite{bdpp}) 
is important and helps us understand Conjecture \ref{f-conj5.13}. 

\begin{thm}\label{f-thm5.16}
Let $X$ be a smooth projective variety. 
Then $X$ is uniruled if and only if $K_X$ is not pseudoeffective. 
\end{thm}

For the reader's convenience, 
we recall 
the definition of uniruled varieties. 

\begin{defn}\label{f-def5.17}
Let $X$ be a smooth projective variety with $\dim X=n$. 
Then $X$ is uniruled if there exist a smooth projective 
variety $Y$ with 
$\dim Y=n-1$ and a dominant rational map $Y\times \mathbb P^1\dashrightarrow 
X$. 
\end{defn}

Therefore, Conjecture \ref{f-conj5.13} says 
that the Kodaira dimension of $X$ is nonnegative 
if $X$ is not covered by rational curves. 
Thus, Conjecture \ref{f-conj5.13} looks very reasonable. 
However, I do not know how to attack it. 

\begin{rem}[Added in September 2016]\label{f-rem5.18}
By Kenta Hashizume's recent result, 
Conjecture \ref{f-conj5.13} implies that 
any projective log canonical pair $(X, \Delta)$ such that 
$K_X+\Delta$ is pseudoeffective 
has a minimal model. For the details, see \cite{hashizume}. 
\end{rem}

\section{Appendix}\label{f-sec6}

In this appendix, we collect some definitions, which 
may help us understand this paper, 
for the 
reader's convenience. 
For the details, see \cite{fujino-fundamental} and \cite{fujino-foundation}. 
Recall that a scheme means a separated scheme of finite type 
over $\mathbb C$ in this paper. 

\begin{say}[Iitaka dimension and Kodaira dimension]\label{f-say6.1}
Let $X$ be a normal projective variety and let $D$ be 
a $\mathbb Q$-Cartier divisor on $X$. 
Then $\kappa (X, D)$ denotes the Iitaka dimension of $D$. 
More precisely, 
$$
\kappa (X, D)=
\underset{m\to \infty}{\lim\sup}\frac{\log \dim H^0(X, \mathcal O_X(
\lfloor mD\rfloor))}
{\log m}.  
$$ 
For the definition of $\lfloor mD \rfloor$, see 
\ref{f-say6.4} below.  
Let $X$ be a smooth projective variety. 
Then we put $\kappa (X)=\kappa (X, K_X)$ and call it 
the Kodaira dimension of $X$. 
\end{say}

\begin{say}[$\mathbb Q$-factorial]\label{f-say6.2}
Let $X$ be a normal variety. 
Then $X$ is called $\mathbb Q$-factorial if 
every prime divisor $D$ on $X$ is $\mathbb Q$-Cartier. 
\end{say}

Note that $\mathbb Q$-factoriality sometimes plays 
crucial roles in the minimal model program. 
The notion of terminal singularities and canonical singularities is 
indispensable in the minimal model program. 

\begin{say}[Terminal singularities and canonical singularities]\label{f-say6.3}
Let $X$ be a normal variety such that $K_X$ is $\mathbb Q$-Cartier. 
If there exists a projective birational morphism 
$f:Y\to X$ from a smooth variety $Y$ such that 
the exceptional locus $\mathrm{Exc}(f)=\sum _{i\in I} E_i$ is a simple 
normal crossing divisor on $Y$. 
In this situation, we can write 
$$
K_Y=f^*K_X+\sum _{i\in I}a_i E_i. 
$$ 
If $a_i>0$ for every $i\in I$, then 
we say that $X$ has only terminal singularities. 
If $a_i\geq 0$ for every $i\in I$, 
then we say that $X$ has only canonical singularities. 
It is well-known that 
$X$ has only rational singularities when $X$ has only 
canonical singularities. 
\end{say}

\begin{say}[Round-down of $\mathbb Q$-divisors]\label{f-say6.4}
Let $D=\sum a_iD_i$ be a $\mathbb Q$-divisor on a normal variety $X$. 
Note that $D_i$ is a prime divisor for every $i$ and 
that $D_i\ne D_j$ for $i\ne j$. 
Of course, $a_i \in \mathbb Q$ for every $i$. 
We put $\lfloor D\rfloor= \sum \lfloor a_i \rfloor D_i$ and 
call it the round-down of $D$. 
For every rational number $x$, $\lfloor x\rfloor$ is the integer defined 
by $x-1<\lfloor x\rfloor \leq x$. 
\end{say}

In the minimal model program, 
we usually use the notion of pairs. 

\begin{say}[Singularities of pairs]\label{f-say6.5}
A pair $(X, \Delta)$ consists of a normal variety $X$ 
and an effective $\mathbb R$-divisor 
$\Delta$ on $X$ such that $K_X+\Delta$ is $\mathbb R$-Cartier. 
A pair $(X, \Delta)$ is called 
kawamata log terminal (resp.~log canonical) if for any 
projective birational morphism 
$f:Y\to X$ from a normal variety $Y$, 
$a(E, X, \Delta)>-1$ (resp.~$\geq -1$) for 
every $E$, where 
$$K_Y=f^*(K_X+\Delta)+\sum _E a(E, X, \Delta)E$$ 
defines $a(E, X, \Delta)$.  
Let $(X, \Delta)$ be a log canonical pair and let $W$ be a 
closed subset of $X$. Then $W$ is called a log canonical 
center of $(X, \Delta)$ if 
there are a projective birational morphism $f:Y\to X$ from a normal 
variety $Y$ and a prime divisor $E$ on $Y$ 
such that $a(E, X, \Delta)=-1$ 
and 
that $f(E)=W$. 
\end{say}

To help understand Theorem \ref{f-thm2.12}, 
Theorem \ref{f-thm2.13}, and Theorem \ref{f-thm3.6}, 
we quickly explain the notion of simple normal crossing pairs and 
some related topics. 

\begin{say}[Simple normal crossing pairs]\label{f-say6.6}
Let $Z$ be a simple normal crossing divisor 
on a smooth 
variety $M$ and let $B$ be an $\mathbb R$-divisor 
on $M$ such that 
$\Supp (B+Z)$ is a simple normal crossing divisor and that 
$B$ and $Z$ have no common irreducible components. 
We put $\Delta_Z=B|_Z$ and consider the pair $(Z, \Delta_Z)$. 
We call $(Z, \Delta_Z)$ a {\em{globally embedded}} simple normal 
crossing pair. 
A pair $(X, \Delta)$ is called a simple normal crossing 
pair if it is Zariski locally isomorphic to 
a globally embedded simple normal crossing 
pair. 
If $(X, \Delta)$ is a simple normal crossing pair and $X$ is a 
divisor on a smooth variety $M$, then 
$(X, \Delta)$ is called an embedded simple normal crossing pair. 
Of course, a globally embedded simple normal crossing pair is 
automatically an embedded simple normal crossing pair. 
\end{say}
\begin{say}[Strata and permissibility]\label{f-say6.7}
Let $(X, \Delta)$ be a simple normal crossing pair. 
Assume that the coefficients of $\Delta$ are in $[0, 1]$. 
Let $\nu:X^\nu \to X$ be the normalization. 
We put $$
K_{X^\nu}+\Theta=\nu^*(K_X+\Delta), 
$$ 
that is, $\Theta$ is the sum of the inverse images of 
$\Delta$ and the singular locus of $X$. 
Then we see that 
$(X^\nu, \Theta)$ is log canonical. 
Let $W$ be a closed subset of $X$. 
Then $W$ is called a stratum of $(X, \Delta)$ if $W$ is 
an irreducible component of $X$ or the $\nu$-image 
of some log canonical center of $(X^\nu, \Theta)$. 
A Cartier divisor $D$ on $X$ is permissible 
with respect to $(X, \Delta)$ if 
$D$ contains no strata of $(X, \Delta)$ in its support. 
A finite $\mathbb R$-linear combination of 
permissible Cartier divisors with respect to $(X, \Delta)$ is called 
a permissible $\mathbb R$-divisor with respect to $(X, \Delta)$. 
\end{say}

We need the notion of nef and log big divisors for Theorem \ref{f-thm2.13}. 

\begin{say}[Nef and log big divisors]\label{f-say6.8} 
Let $f:(Y, \Delta)\to X$ be a proper morphism 
from a simple normal crossing pair $(Y, \Delta)$ to 
a scheme $X$. Assume that the coefficients of $\Delta$ are in $[0, 1]$. 
Let $\pi:X\to V$ be a proper morphism 
between schemes. 
Let $H$ be an $\mathbb R$-Cartier divisor 
on $X$. 
Then $H$ is nef and log big over $V$ with 
respect to 
$f:(Y, \Delta)\to X$ if 
$H$ is nef over $V$ and $H|_{f(W)}$ is big over 
$\pi\circ f(W)$ for every stratum $W$ of $(Y, \Delta)$. 
Note that if $H$ is ample over $V$ then 
$H$ is nef and log big over $V$ with respect to 
$f:(Y, \Delta)\to X$. 
\end{say}

For Theorem \ref{f-thm2.12} and Theorem \ref{f-thm2.13}, 
let us recall the definition of $\mathbb R$-linear equivalence. 

\begin{say}[$\mathbb R$-divisors]
\label{f-say6.9}
Let $B_1$ and $B_2$ be two $\mathbb R$-Cartier divisors 
on a scheme $X$. 
Then $B_1$ is linearly (resp.~$\mathbb Q$-linearly, 
or $\mathbb R$-linearly)
equivalent to 
$B_2$, denoted by $B_1\sim B_2$ 
(resp.~$B_1\sim _{\mathbb Q} B_2$, or 
$B_1\sim _{\mathbb R}B_2$) if 
$$
B_1=B_2+\sum _{i=1}^k r_i (f_i)
$$ 
such that $f_i \in \Gamma (X, \mathcal K_X^*)$ and $r_i\in \mathbb Z$ (resp.~$r_i \in \mathbb Q$, or $r_i \in 
\mathbb R$) for every $i$. Here, $\mathcal K_X$ is 
the sheaf of total quotient rings of 
$\mathcal O_X$ and $\mathcal K_X^*$ 
is the sheaf of invertible elements in the sheaf of rings $\mathcal K_X$. 
We note that 
$(f_i)$ is a principal Cartier divisor 
associated to $f_i$, that is, 
the image of $f_i$ by 
$
\Gamma (X, \mathcal K_X^*)\to\Gamma (X, \mathcal K_X^*/\mathcal O_X^*)$, 
where $\mathcal O_X^*$ is the sheaf of invertible elements in $\mathcal O_X$. 

Let $f:X\to Y$ be a morphism between schemes. 
If there is an $\mathbb R$-Cartier divisor $B$ on $Y$ such that 
$$
B_1\sim _{\mathbb R}B_2+f^*B, 
$$ 
then $B_1$ is said to be relatively $\mathbb R$-linearly equivalent to 
$B_2$. It is denoted by $B_1\sim _{\mathbb R, f}B_2$ or 
$B_1\sim _{\mathbb R, Y}B_2$. 
\end{say}


\end{document}